\documentclass{amsart}

\usepackage{amssymb,amsfonts, amsmath, amsthm}
\usepackage[all,arc]{xy}
\usepackage{enumerate}
\usepackage{mathrsfs}
\usepackage{mathtools}
\usepackage[mathscr]{euscript}
\usepackage{array}
\usepackage{appendix}

\usepackage{tikz}
\usepackage{mathpazo}
\usepackage{tikz-cd}
\usepackage{textcomp}
\usetikzlibrary{decorations.pathmorphing,shapes}
\usepackage[pagebackref, colorlinks=true, linkcolor=blue, citecolor = red]{hyperref}
 \renewcommand*{\backref}[1]{}
 \renewcommand*{\backrefalt}[4]{({%
     \ifcase #1 Not cited.%
           \or On p.~#2%
           \else On pp.~#2%
     \fi%
     })}

\usepackage[nameinlink,capitalise,noabbrev]{cleveref}
\crefname{subsection}{Subsection}{Subsection}

\usetikzlibrary{patterns}

\usepackage{geometry}
\usepackage{todonotes}

\newcommand{\s}{\mathscr{S}}
\renewcommand{\ss}{s\mathscr{S}}
\newcommand{\sss}{ss\mathscr{S}}

\newcommand{\C}{\mathscr{C}}

\newcommand{\bI}{\mathbb{I}}

\newcommand{\D}{\mathscr{D}}

\newcommand{\F}{\mathscr{F}}
\newcommand{\V}{\mathscr{V}}

\newcommand{\cM}{\mathcal{M}}
\newcommand{\M}{\mathscr{M}}
\newcommand{\N}{\mathcal{N}}

\newcommand{\R}{\mathscr{R}}

\newcommand{\Diag}{\mathrm{Diag}}

\newcommand{\id}{\mathrm{id}}
\newcommand{\ds}{\displaystyle}

\newcommand{\Hom}{\mathrm{Hom}}
\newcommand{\Map}{\mathrm{Map}}

\newcommand{\Fun}{\mathrm{Fun}}

\renewcommand{\exp}[2]{\mathrm{exp}(#1,#2)}

\newcommand{\sint}{s\hspace{-0.04in}\int}

\newcommand{\sT}{s\mathscr{T}}
\newcommand{\bT}{\mathbb{T}}
\newcommand{\sH}{s\mathscr{H}}

\newcommand{\sbT}{s\bT}
\newcommand{\sbI}{s\bI}

\newcommand{\ssbI}{s\sbI}
\newcommand{\ssbT}{s\sbT}
\newcommand{\ssint}{s\sint}

\newcommand{\ssT}{s\sT}
\newcommand{\ssH}{s\sH}

\newcommand{\ordered}[1]{< #1 >}

\newcommand{\phiDiag}{\Diag_1}

\newcommand{\LFib}{\mathscr{L}\mathscr{F}\mathrm{ib}}
\newcommand{\Val}{\mathscr{V}\mathrm{al}}
\newcommand{\fDiag}{\mathscr{D}\mathrm{iag}}
\newcommand{\LEmb}{\mathscr{L}\mathscr{E}\mathrm{mb}}
\newcommand{\VEmb}{\mathscr{V}\mathscr{E}\mathrm{mb}}

\newcommand{\iF}{p_1}
\newcommand{\iphi}{p_2}
\newcommand{\Fib}{\mathscr{F}\mathrm{ib}}

\newcommand{\setref}[1]{\cref{Subsec Simplicial Sets}(\ref{#1})}
\newcommand{\spaceref}[1]{\cref{Subsec Simplicial Spaces}(\ref{#1})}
\newcommand{\leftref}[1]{\cref{Subsec A Reminder on the Covariant Model Structure}(\ref{#1})}

\newcommand{\adjun}[4]{
\begin{tikzcd}[row sep=0.5in, column sep=0.5in]
 #1  \arrow[r, shift left=1.8, "#3", "\bot"'] \pgfmatrixnextcell
 #2 \arrow[l, shift left=1.8, "#4"] 
\end{tikzcd}
}

\newcommand{\comsq}[8]{
  \begin{tikzcd}[row sep=0.25in, column sep=0.25in]
    #1 \arrow[r, "#5"] \arrow[d, "#6"']
    \pgfmatrixnextcell #2 \arrow[d, "#7"] \\
    #3 \arrow[r, "#8"]
    \pgfmatrixnextcell #4
  \end{tikzcd}
}

\newcommand{\pbsq}[8]{
  \begin{tikzcd}[row sep=0.25in, column sep=0.25in]
    #1 \arrow[r, "#5"] \arrow[d, "#6"'] \arrow[dr, phantom, "\ulcorner", very near start]
    \pgfmatrixnextcell #2 \arrow[d, "#7"] \\
    #3 \arrow[r, "#8"']
    \pgfmatrixnextcell #4
  \end{tikzcd}
}

\newtheorem{theone}[equation]{Theorem}
\newtheorem{lemone}[equation]{Lemma}
\newtheorem{propone}[equation]{Proposition}
\newtheorem{corone}[equation]{Corollary}

\theoremstyle{definition}
\newtheorem{defone}[equation]{Definition}
\newtheorem{exone}[equation]{Example}

\theoremstyle{remark}
\newtheorem{remone}[equation]{Remark}
\newtheorem{notone}[equation]{Notation}

\newtheoremstyle{TheoremNum}
{}{}              
{\itshape}                      
{}                              
{\bfseries}                     
{.}                             
{ }                             
{\thmname{#1}\thmnote{ \bfseries #3}}
\theoremstyle{TheoremNum}

\numberwithin{equation}{section}

\makeatletter
\def\@seccntformat#1{%
  \expandafter\ifx\csname c@#1\endcsname\c@section\else
  \csname the#1\endcsname\quad
  \fi}
\makeatother   

\title{Cartesian Fibrations of Complete Segal Spaces}

\author{Nima Rasekh}
\address{{\'E}cole Polytechnique F{\'e}d{\'e}rale de Lausanne, SV BMI UPHESS, Station 8, CH-1015 Lausanne, Switzerland}
\email{nima.rasekh@epfl.ch}

\date{February 2021}

\begin{document}

\begin{abstract}
 Cartesian fibrations were originally defined by Lurie in the context of quasi-categories and are commonly used in $(\infty,1)$-category theory to study presheaves valued in $(\infty,1)$-categories. 
 In this work we define and study fibrations modeling presheaves valued in simplicial spaces and their localizations. This includes defining a model structure for these fibrations and giving effective tools to recognize its fibrations and weak equivalences.
 This in particular gives us a new method to construct Cartesian fibrations via complete Segal spaces. In addition to that, it allows us to define and study fibrations modeling presheaves of Segal spaces. 
\end{abstract}

\maketitle
\addtocontents{toc}{\protect\setcounter{tocdepth}{1}}

\tableofcontents

\section{Introduction}\label{Sec Introduction}

\subsection{Cartesian Fibrations of Quasi-Categories}
The {\it theory of $(\infty,1)$-categories} helped formalize the notion of {\it homotopies} that first arose in classical algebraic topology.
This helped overcome many early challenges in algebraic topology.
For example it helped develop a homotopy invariant notion of colimit, making sense of {\it homotopy colimits} \cite{bousfieldkan1972yellowmonster}, or helped properly define a {\it smash product of spectra} \cite{gepnergrothnikolaus2015infiniteloopspacemachine}, an important problem in the early days of stable homotopy theory \cite{adams1995stable,ekmm1995stable}.
More generally, it created a foundation for properly developing {\it ``homotopy coherent mathematics"}, which has now found applications in many branches of mathematics, such as algebraic geometry \cite{lurie2018sag}, differential geometry \cite{nikolausschreiberstevenson2015principaltheory,nikolausschreiberstevenson2015principalpres} and even mathematical physics \cite{lurie2009cobordism}.
As one might expect such benefits also come with a price. For example, it greatly complicates the notion of functoriality, which now needs to be homotopy coherent and hence requires checking an infinite number of conditions.

Fortunately, certain important classes of functors can be defined in alternative ways, that are often easier to construct in practice. For example functors out of an $(\infty,1)$-category into the $(\infty,1)$-category of spaces are equivalent to {\it left fibrations} over that $(\infty,1)$-category. This was first observed by Joyal who was developing the category theory of {\it quasi-categories}, a popular model of $(\infty,1)$-categories \cite{joyal2008notes,joyal2008theory}. It was then further studied by Lurie, who also presented one of the first proofs of the equivalence between functors and fibrations in the context of quasi-categories \cite{lurie2009htt}. In the subsequent years many authors have reviewed the theory of left fibrations and its relation with functors from many different perspectives: There are alternative methods for defining the model structure for left fibrations, the {\it covariant model structure}, in the context of quasi-categories \cite{nguyen2019covariant}. Moreover, there are now many alternative proofs of the equivalence between left fibrations and functors again in the quasi-categorical context \cite{heutsmoerdijk2015leftfibrationi,heutsmoerdijk2016leftfibrationii,stevenson2017covariant,cisinski2019highercategories}.
There are also studies of left fibrations using {\it complete Segal spaces} \cite{rezk2001css}, another model of $(\infty,1)$-categories \cite{debrito2018leftfibration,rasekh2017left,kazhdanvarshvsky2014yoneda}.  Moreover, there is an analysis of left fibrations in the context of an $\infty$-cosmos, which is a model-independent approach to $(\infty,1)$-category theory using various ideas from $2$-category theory \cite{riehlverity2018elements}. Finally, left fibrations have also been studied in a homotopy type theoretical context \cite{riehlshulman2017rezktypes}.

Another class of functors that can studied via fibrations are functors valued in $(\infty,1)$-categories themselves. Here the corresponding fibrations are known as {\it coCartesian fibrations}. These were first defined by Lurie \cite{lurie2009htt}, who proved an equivalence between fibrations and functors by constructing a Quillen equivalence between appropriately defined model categories. However, coCartesian fibrations have not received the same attention that left fibrations have. There has been interesting work on the model-independent aspects of coCartesian fibrations, both from a quasi-categorical perspective \cite{mazelgee2019cartfib,ayalafrancis2020fibrations} as well as from an $\infty$-cosmos perspective \cite{riehlverity2018elements}. However, the {\it coCartesian model structure} and its equivalence with functors in \cite{lurie2009htt} have not been tackled again in the quasi-categorical setting, let alone other models of $(\infty,1)$-categories.

There are several complicating factors that have contributed to our current predicament. One very mysterious issue that arises when studying coCartesian fibrations is that although quasi-categories are simplicial sets, the model structure for coCartesian fibrations has only been defined for {\it marked simplicial sets} and it is widely believed that it is not possible to define an appropriate model structure on simplicial sets that can help us study coCartesian fibrations. This technicality adds a layer, in particular as the category of marked simplicial sets is not a presheaf category hence depriving us of many techniques to study fibrations (that for example play an important role in \cite{cisinski2019highercategories,nguyen2019covariant}). Another complicating factor comes from the fact that functors into $(\infty,1)$-categories have an inherent {\it $(\infty,2)$-categorical} character (as we can talk about natural transformations of such functors) and while there are several models of $(\infty,2)$-categories in the literature \cite{rezk2010thetanspaces,barwick2005nfoldsegalspaces,verity2008complicial,ara2014highersegal}, the study of its category theory and in particular fibrations is still in its early stages \cite{lurie2009goodwillie,gagnaharpazlanari2020inftytwofibrations}.

\subsection{Cartesian Fibrations via Complete Segal Objects}
Up to this point we discussed the importance of coCartesian fibrations and the need to find alternative perspectives.
The goal of this paper is to offer one such alternative perspective using {\it complete Segal objects} (also called {\it Rezk objects} \cite{riehlverity2017inftycosmos}). 
Before going into further details it is instructive to review the construction of complete Segal spaces, due to Rezk \cite{rezk2001css}, which goes as follows:
\begin{itemize}
	\item[(1)] He starts with the category of simplicial sets with the Kan model structure, giving us a model for spaces.
	\item[(2)] He then takes simplicial diagrams $X$ in spaces, defining {\it simplicial spaces}, and gives that the Reedy model structure.
	\item[(3)] He adds two restrictions by using Bousfield localization on the Reedy model structure:
	\begin{itemize}
		\item[(I)] {\it Segal Condition:} A Reedy fibrant simplicial space $X$ is a {\it Segal space} if the map 
		$$X_n \to X_1 \times_{X_0} ... \times_{X_0} X_1$$
		for $n \geq 2$ is a Kan equivalence.
		\item[(II)] {\it Completeness Condition:} A Segal space $X$ is a {\it complete Segal space} if the map
		$$X_{hoequiv} \to X_1$$
		is a Kan equivalence, where $X_{hoequiv}$ can be described as a finite limit \cite[10]{rezk2010thetanspaces}.
	\end{itemize}
	For a review of complete Segal spaces see \cref{Subsec Complete Segal Spaces}.
\end{itemize}  
Later complete Segal spaces model structure was proven to be equivalent to the model structure for quasi-categories \cite{joyaltierney2007qcatvssegal}
and other models of $(\infty,1)$-categories \cite{toen2005unicity,bergner2010survey}.

The beauty of the complete Segal space approach to $(\infty,1)$-categories is that the process we outlined above can be generalized from spaces to any $(\infty,1)$-category with finite limits, giving us a notion of complete Segal objects. In particular we can apply the process to the $(\infty,1)$-category of functors valued in spaces, or, equivalently, to left fibrations. 
Applying the process to the $(\infty,1)$-category of space-valued functors evidently results in the $(\infty,1)$-category of $(\infty,1)$-category-valued functors. This naturally motivates studying complete Segal objects of left fibrations as a fibrational analogue to functors valued in $(\infty,1)$-categories hence suggest following approach to coCartesian fibrations:
\begin{center}
	{\bf coCartesian fibrations are complete Segal objects in left fibrations.}
\end{center}

\subsection{Cartesian Fibration of Complete Segal Spaces}
 In order to start the process from left fibrations to coCartesian fibrations we first need to choose a model for our left fibrations. 
 Here we will use left fibrations of simplicial spaces as studied in \cite{rasekh2017left}, however, it should be noted that using left fibrations as studied by Lurie \cite{lurie2009htt} would give us the same results. In fact the equivalence of the two resulting coCartesian model structures has been proven in the follow-up work \cite{rasekh2021cartfibmarkedvscso}.
 
 Having decided which model of left fibrations to use we define coCartesian fibrations and study their properties simply by following the same three steps that Rezk used:
 \begin{itemize}
 	\item[(1)] Start with the category of simplicial spaces over a fixed simplicial space and give it a model structure such that the fibrant objects are the left fibrations. This model structure is known as the {\it covariant model structure} \cite{rasekh2017left} and is reviewed in \cref{Subsec A Reminder on the Covariant Model Structure}.
 	\item[(2)] Take simplicial diagrams in left fibrations. Then give the resulting category a Reedy model structure, calling it the {\it Reedy covariant model structure on bisimplicial spaces}. Then observe how the properties of the covariant model structure transfers to the Reedy covariant model structure. This is the content of  \cref{Sec The Reedy Covariant Model Structure}.
 	\item[(2.5)] Next we will do a general analysis how the Reedy covariant model structure behaves when we use Bousfield localizations, in particular analyzing the fibrant objects (\cref{cor:localized left}) and weak equivalences (\cref{the:recognition principle for localized Reedy right equivalences}) and studying its invariance (\cref{the:loc Ree contra invariant under CSS}). This is the goal of \cref{Sec Localizations of Reedy Right Fibrations}.
 	\item[(3)] Finally, we apply the results of the previous section and focus on the particular conditions associated with Segal spaces and complete Segal spaces to define Segal coCartesian fibrations and coCartesian fibrations. 
    We will cover that in \cref{sec:cartesian Fibrations}.
 \end{itemize}
 Note we can prove that the resulting model structure is Quillen equivalent to the Cartesian model structure on marked simplicial sets defined by Lurie. That is the main result of the follow up work \cite{rasekh2021cartfibmarkedvscso}.
  
\subsection{Why Complete Segal Space Approach?}
 Given that we already had a Cartesian model structure, why present an alternative way? Beside a theoretical satisfaction of approaching an interesting topic from a new angle, there are also concrete benefits:
 \begin{enumerate}
 	\item {\bf Exposition:} Cartesian fibrations are notoriously difficult to understand. The main source for many results is still \cite[Chapter 3]{lurie2009htt} and is quite technical, requiring a lot of background knowledge. This makes it difficult for most, except for a small number of  experts, to use Cartesian fibrations to prove new results.
 	
 	The complete Segal object approach to Cartesian fibrations requires far less theoretical background. It primarily relies on understanding left fibrations, which are in fact easier and have been studied by many different people meaning there are now excellent resources for mathematicians interested in fibrations. 
 	
 	\item {\bf Direct Proofs:} One important implication of the equivalence between left fibrations and space-valued functors is the fact that the covariant model structure is invariant under $(\infty,1)$-categorical equivalences, meaning that an equivalence in the model structure for quasi-categories gives us a Quillen equivalence of covariant model structures \cite[Remark 2.1.4.11]{lurie2009htt}. However, there are now also direct proofs of this fact that only use structural properties of the covariant model structure, both in the setting of quasi-categories \cite{heutsmoerdijk2015leftfibrationi} as well as complete Segal spaces \cite{rasekh2017left}.
 	
 	Generalizing to coCartesian fibrations, we can still use the equivalence with functors to deduce it is invariant under $(\infty,1)$-categorical equivalences, as has been done in \cite[Proposition 3.3.1.1]{lurie2009htt}. However, we do not have a direct proof only using the coCartesian model structure on marked simplicial sets. On the other hand, using the complete Segal approach to coCartesian fibrations allows us to generalize the invariant proof for left fibrations to coCartesian fibrations in a reasonably straightforward manner (as shown in \cref{the:cart invariant}).
	
 	\item {\bf Segal coCartesian Fibrations:}
 	{\it Homotopy type theory} is a new approach to the foundations that is inherently homotopy invariant \cite{hottbook2013}. 
 	It has opened the possibility of finding a model independent approach to $(\infty,1)$-category theory. As one would expect of an axiomatic system, one important question is their expressiveness, meaning which axioms are necessary to prove which result. An important example is the {\it univalence axiom} which is simply the type theoretic articulation of the completeness condition we use to define complete Segal spaces. 
 	For example, in their paper \cite{riehlshulman2017rezktypes} Riehl and Shulman introduce a notion of $(\infty,1)$-category, a {\it Rezk type}, inside their type theory. They then prove that the Yoneda lemma holds without the univalence axiom, whereas equality of various notions of adjunctions does require univalence.
 	
 	In the language of $(\infty,1)$-categories trying to tackle this question corresponds to proving a result for a general Segal spaces vs. observing that the completeness condition is in fact necessary. For example, the independence of univalence from the Yoneda lemma corresponds to proving the Yoneda lemma for Segal spaces, which in fact has been done (independently) in \cite{rasekh2017left}. In order to seriously pursue such question we would need a notion of fibration for Segal spaces, similar to coCartesian fibrations. Defining such fibrations does not seem possible using marked simplicial sets, whereas we can do so using the complete Segal approach (\cref{sec:cartesian Fibrations}). Hence, using the complete Segal approach to coCartesian fibration allows us to tackle more general question of interest related to foundations and the necessity of completeness.
 	
 	\item {\bf Representable Cartesian Fibrations:}
 	One important class of left fibrations are {\it representable left fibrations}, which are precisely the fibrations that correspond to corepresentable functors. These left fibrations play an extraordinary role in $(\infty,1)$-category theory and many important results (such as limits, adjunctions, ...) can be reduced to determining the representability of certain left fibrations. 
 	
 	We can similarly try to determine when a coCartesian fibration is representable by a simplicial object. While it is theoretically possible to study such coCartesian fibrations using the marked simplicial approach (as has been done in \cite{stenzel2020comprehension}), the complete Segal approach is perfectly tailored to tackle such questions. The study of such representable coCartesian fibrations deserves its own attention and hence is part of a follow up to this paper \cite{rasekh2017cartesian}.

 	\item {\bf Fibrations of $(\infty,n)$-Categories:}
 	The same way that $(\infty,1)$-category theory has led to a precise notion of ``weak $1$-categories", the development of $(\infty,n)$-categories is helping us conceptualize {\it weak $n$-categories}. Though in its early stages it has already contributed to the advancement of topological field theories \cite{lurie2009cobordism,calaquescheimbauer2019cobordism}, derived algebraic geometry \cite{gaitsgoryrozenblyum2017dagI,gaitsgoryrozenblyum2017dagII} and $(\infty,1)$-category theory itself \cite{lurie2009goodwillie}. Further applications and studies require a good theory of fibrations. 
 	
    Some common models of $(\infty,n)$-categories, such as $\Theta_n$-spaces \cite{rezk2010thetanspaces} and $n$-fold complete Segal spaces \cite{barwick2005nfoldsegalspaces} are in fact direct (and equivalent \cite{bergnerrezk2013comparisoni,bergnerrezk2020comparisonii}) generalizations of complete Segal spaces. As the complete Segal object to fibrations is inherently inductive, it opens the possibility of defining fibrations for $(\infty,n)$-categories by simply choosing appropriate $\Theta_n$-diagrams or $n$-fold simplicial diagrams in left fibrations.
 \end{enumerate}

\subsection{Relation to Other Work} 
This paper is the first part of a three-paper series which introduces the bisimplicial approach to Cartesian fibrations: 
\begin{enumerate}
	\item {\bf Cartesian Fibrations of Complete Segal Spaces}
	\item {\bf Quasi-Categories vs. Segal Spaces: Cartesian Edition} \cite{rasekh2021cartfibmarkedvscso}
	\item {\bf Cartesian Fibrations and Representability} \cite{rasekh2017cartesian}
\end{enumerate}
In particular, the second paper proves that the approach here coincides with the approach via marked simplicial sets. 
The third paper gives an application of the bisimplicial approach to the study of representable Cartesian fibrations. 

\subsection{Acknowledgments} \label{Subsec Acknowledgements}
I want to thank my advisor Charles Rezk who suggested this topic to me.
I also want to thank the referee of the paper \cite{rasekh2017left} for suggestions that have also resulted in many improvements of this work.

\section{Reviewing Concepts} \label{Sec Reminders Notation}
 In this section we review some basic concepts regarding model categories, simplicial spaces, complete Segal spaces, left fibrations and bisimplicial spaces that we will need in the coming sections. 
 
 \subsection{Model Categories} \label{subsec:model categories}
 We will use the language of model categories throughout and so use results from \cite{hovey1999modelcategories,hirschhorn2003modelcategories,lurie2009htt,joyaltierney2007qcatvssegal}.
 Here we will only state few results explicitly.
 
 \begin{remone} \label{rem:jt calculus}
 	Recall a model structure $\cM$ on a category $\C$ is called {\it compatible with Cartesian closure} if for cofibrations $i,j$ and fibration $p$, the {\it pushout-product} $i \square j$  is a cofibration and the {\it pullback-exponential} $\exp{i}{p}$ is a fibration, which is trivial if either maps involved are trivial. 
 \end{remone}
  
  For more details pushout products and pullback exponentials and their interaction (also known as {\it Joyal-Tierney calculus}) see the original source \cite[Section 7]{joyaltierney2007qcatvssegal} or \cite[Subsection 2.1]{rasekh2017left}. We also need a result guaranteeing that Bousfield localizations preserve Quillen equivalences.
  
 \begin{theone} \label{the:localization of Quillen equiv}
 	Let $\C$ and be $\D$ two categories and $\cM$ and $\N$ two simplicial, combinatorial model structures such that the cofibrations are inclusions in $\C$ and $\D$ respectively.
 	Moreover, 
 	\begin{center}
 		\adjun{\C^\cM}{\D^\N}{F}{G}
 	\end{center}
 	be a simplicial Quillen adjunction (equivalence) of model structures and $S$ a set of cofibrations in $\cM$. 
 	Then we get a Quillen adjunction (equivalence)
 	\begin{center}
 		\adjun{\C^{\cM_S}}{\D^{\N_{F(S)}}}{F}{G}
 	\end{center}
 	where the left hand side is the localized model structure with respect to $S$ and the right hand side has been localized with respect to $F(S)$.
 \end{theone}
 
 \begin{proof}
 	First we assume $(F,G)$ is a Quillen adjunction between the $\cM$ and $\N$ model structure and prove it is a Quillen adjunction between the $\cM_S$ and $\N_{F(S)}$ model structure. We know that $F$ preserves cofibrations, hence, by \cite[Corollary A.3.7.2]{lurie2009htt} it suffices to check that the right adjoint $G$ preserves fibrant objects. 
 	Let $Y$ in $\D$ be $\N_{F(S)}$-fibrant. Then $GY$ is $\cM$-fibrant and so we only need to prove that for all maps $f: A \to B$ in $S$
 	\begin{equation}\label{eq:appendix eq}
 		f^*:\Map_\C(B,GY) \to \Map_\C(A,GY)
 	\end{equation}
 	is a Kan equivalence. By adjunction this is equivalent to 
 	\begin{equation} \label{eq:appendix eqtwo}
 		F(f)^*: \Map_\D(FB,Y) \to \Map_\D(FA,Y)
 	\end{equation}
 	being an equivalent, which holds by assumption.
 	Notice, we can use the same argument to deduce that if an object $Y$ in $\D$ is $\N$-fibrant, such that $G(Y)$ is $\cM_S$-fibrant, then $Y$ is in fact $\N_{F(S)}$-fibrant. Indeed, $\N_{F(S)}$-fibrancy implies the map \ref{eq:appendix eq} is an equivalence which implies that \ref{eq:appendix eqtwo} is an equivalence giving us the desired result.
 	
 	Next we want to prove that if $(F,G)$ is a Quillen equivalence between the $\cM$ and $\N$-model structures then it is also a Quillen equivalence between the $\cM_S$ and $\N_{F(S)}$ model structure. 
 	First, observe the derived counit map is an equivalence. 
 	Indeed, all objects are cofibrant, which means the derived counit map is just the counit map, which by assumption is an equivalence in $\N$ and hence in $\N_{F(S)}$.
 	
 	Next we show the derived unit map is an equivalence. Let $X$ be an $\cM_{F(S)}$-bifibrant object in $\C$. Let $R(F(X))$ be an $\N$-fibrant replacement of $F(X)$. Then $R(F(X))$ is in fact $\N_{F(S)}$-fibrant and hence an $\N_{F(S)}$-fibrant replacement. 
 	Indeed, by the previous paragraph it suffices to prove that $G(R(F(X))$ is $\cM_S$-fibrant. However, as $(F,G)$ is a Quillen equivalence between the $\cM$ and $\N$ model structure, it is equivalence to $X$ (via the derived unit map) and hence is $\cM_S$-fibrant by assumption. Hence $X \to G(R(F(X)))$ is in fact the derived unit map in the $\cM_S$ model structure. By assumption it is an $\cM$-equivalence and so it is also an $\cM_S$-equivalence, finishing the proof.
 \end{proof}
  
 \subsection{Simplicial Sets} \label{Subsec Simplicial Sets}
 $\s$ will denote the category of simplicial sets, which we will call {\it spaces}.
 We will use the following notation with regard to spaces:
 \begin{enumerate} 
  \item $\Delta$ is the indexing category with objects posets $[n] =  \{ 0,1,...,n \} $ and mappings maps of posets.
  \item \label{item:morphism notation}We will denote a morphism $[n] \to [m]$ by a sequence of numbers $\ordered{a_0, ... ,a_n}$, where $a_i$ is the image of $i \in [n]$.
  \item $\Delta[n]$ denotes the simplicial set representing $[n]$ i.e. $\Delta[n]_k = \Hom_{\Delta}([k], [n])$. 
  \item $\partial \Delta[n]$ denotes the boundary of $\Delta[n]$ i.e. the largest sub-simplicial set which does not include $id_{[n]}: [n] \to [n]$.
  \item \label{item:Jn} Let $I[l]$ be the category with $l$ objects and one unique isomorphisms between any two objects.
  Then we denote the nerve of $I[l]$ as $J[l]$. 
  It is a Kan fibrant replacement of $\Delta[l]$ and comes with an inclusion $\Delta[l] \rightarrowtail  J[l]$,
  which is a Kan equivalence. 
 \end{enumerate}

 \subsection{Simplicial Spaces} \label{Subsec Simplicial Spaces}  
 $\ss = \Fun(\Delta^{op}, \s) $ denotes the category of simplicial spaces (bisimplicial sets). 
 We have the following basic notations with regard to simplicial spaces:
 \begin{enumerate}
  \item \label{item:embed} We embed the category of spaces inside the category of simplicial spaces as constant simplicial spaces 
  (i.e. the simplicial spaces $S$ such that $S_n = S_0$ for all $n$ and all simplicial operator maps are identities). 
  \item More generally we say a simplicial space is {\it homotopically constant} if all simplicial operator maps $X_n \to X_m$
  are equivalences (and in particular $X_n$ are all equivalent to $X_0$).
  \item \label{item:Fn} Denote $F(n)$ to be the simplicial space defined as $F(n)_{kl} = \Delta[n]_k = \Hom_{\Delta}([k],[n])$. 
  Moreover  $\partial F[n]$ denotes the boundary of $F(n)$. 
  \item \label{item:En} Denote $E(n)$ to be the simplicial space defined as $E(n)_{kl} = J[n]_k$, as defined in \setref{item:Jn}.
  \item The category $\ss$ is enriched over spaces
  $$\Map_{\ss}(X,Y)_n = \Hom_{\ss}(X \times \Delta[n], Y).$$
  \item The category $\ss$ is also enriched over itself
  $$(Y^X)_{kn} = \Hom_{\ss}(X \times F(n) \times \Delta[l], Y).$$
  \item By the Yoneda lemma, for a simplicial space $X$ we have a bijection of spaces
  $$X_n \cong \Map_{\ss}(F(n),X).$$
  \end{enumerate}

  \subsection{Reedy Model Structure} \label{Subsec Reedy Model Structure}
 The category of simplicial spaces has a Reedy model structure \cite{reedy1974modelstructure}, which is defined as follows:
 \begin{itemize}
  \item[(F)] A map $f: Y \to X$ is a (trivial) Reedy fibration if for each $n \geq 0$ the following map of spaces is a (trivial) Kan fibration
  $$ \Map_{\ss}(F(n),Y) \to \Map_{\ss}(\partial F(n),Y) \underset{\Map_{\ss}(\partial F(n), X)}{\times} \Map_{\ss}(F(n), X).$$
  \item[(W)] A map $f:Y \to X$ is a Reedy equivalence if it is a level-wise Kan equivalence.
  \item[(C)] A map $f:Y \to X$ is a Reedy cofibration if it is a monomorphism.
 \end{itemize}
 The Reedy model structure is very helpful as it enjoys many features that can help us while doing computations.
 In particular, it is {\it combinatorial}, {\it simplicial} and {\it proper}. Moreover, it is also {\it compatible with Cartesian closure} (\cref{rem:jt calculus}). These properties in particular imply that we can apply Bousfield localizations to the Reedy model structure. 
 See \cite{hirschhorn2003modelcategories} for more details.
 
 \subsection{Complete Segal Spaces} \label{Subsec Complete Segal Spaces}
 The Reedy model structure can be localized such that it models $(\infty,1)$-categories \cite{rezk2001css}.
 This is done in two steps.
 First we define {\it Segal spaces}. A Reedy fibrant simplicial space $X$ is called a {\it Segal space} if the map
 $$ \Map(F(n),X) \xrightarrow{\ \ \simeq \ \ } \Map(G(n),X)$$ 
 induced by the {\it spine inclusion} $G(n) \hookrightarrow F(n)$ is a Kan equivalence for $n \geq 2$ \cite[Section 5]{rezk2001css}.
 Segal spaces come with a model structure.
 
 \begin{theone} \label{The Segal Space Model Structure}
  \cite[Theorem 7.1]{rezk2001css}
  There is a unique combinatorial left proper simplicial model structure on the category $\ss$ of simplicial spaces
  called the Segal space model category structure, and denoted $\ss^{Seg}$, with the following properties.
  \begin{enumerate}
   \item The cofibrations are the monomorphisms.
   \item The fibrant objects are the Segal spaces.
   \item The weak equivalences are the maps $f$ such that $\Map_{\ss}(f, W)$ is
   a weak equivalence of spaces for every Segal space $W$.
  \end{enumerate}
 \end{theone}

 Segal spaces do not give us a model of $(\infty,1)$-categories. For that we need {\it complete Segal spaces}.
 A Segal space is called a {\it complete Segal space} if the map 
 $$\Map(E(1),W) \to \Map(F(0),W)$$
 induced by the inclusion $F(0) \to E(1)$ (\spaceref{item:En}) is a Kan equivalence.
 Complete Segal spaces come with their own model structure, the {\it complete Segal space model structure}.
 
 \begin{theone} \label{The Complete Segal Space Model Structure}
  \cite[Theorem 7.2]{rezk2001css}
  There is a unique combinatorial left proper simplicial model structure on the category $\ss$ of simplicial spaces, 
  called the complete Segal space model category structure, and denoted $\ss^{CSS}$, with the following properties.
  \begin{enumerate}
   \item The cofibrations are the monomorphisms.
   \item The fibrant objects are the complete Segal spaces.
   \item The weak equivalences are the maps $f$ such that $\Map_{\ss}(f, W)$ is a weak equivalence of spaces for every complete Segal space $W$.
  \end{enumerate}
 \end{theone}
 
 A complete Segal space is a model for an $(\infty,1)$-category.
 For a better understanding of complete Segal spaces see \cite[Section 6]{rezk2001css} and for a comparison with other models see 
 \cite{joyaltierney2007qcatvssegal,bergner2010survey}.

\subsection{A Reminder on the Covariant Model Structure} \label{Subsec A Reminder on the Covariant Model Structure}
This section will serve as a short reminder on the covariant model structure and all of its relevant definitions and theorems.
For more details see \cite{rasekh2017left}, where all these definitions and theorems are discussed in more detail.

A Reedy fibration $p: L \to X$ ($q: R \to X$) is called a {\it left fibration} ({\it right fibration}) if the following is a homotopy pullback square (using the notation \setref{item:morphism notation})
 \begin{center}
	\pbsq{L_n}{X_n}{L_0}{X_0}{\ordered{0}^*}{p_n}{p_0}{\ordered{0}^*}, \
	\pbsq{R_n}{X_n}{R_0}{X_0}{\ordered{n}^*}{q_n}{q_0}{\ordered{n}^*}.
\end{center}
 
 Left fibrations (right fibrations) come with a model structure that has many desirable properties:
 There is unique left proper combinatorial simplicial model structure on the over category $\ss_{/X}$, called the {\it covariant model structure} ({\it contravariant model structure}). 
 Here we will only state the relevant properties of the covariant model structure:
 \begin{enumerate}
 	\item \label{leftitem:left fib} \cite[Theorem 3.12]{rasekh2017left} The fibrant object are left fibrations.
 	\item \label{leftitem:local left fib} \cite[Lemma 3.10]{rasekh2017left} For a Reedy fibration $p: Y \to X$, the following are equivalent
 	\begin{enumerate}
 		\item[(I)] $p$ is a left fibration.
 		\item[(II)] For every map $\sigma: F(n) \times \Delta[l] \to X$ the induced map 
 		$\sigma^*Y \to F(n) \times \Delta[l]$ is a left fibration.
 		\item[(III)] For every map $\sigma: F(n) \to X$ the induced map 
 		$\sigma^*Y \to F(n)$ is a left fibration.
 	\end{enumerate}
    \item \label{leftitem:equiv left fib} \cite[Theorem 4.34]{rasekh2017left} For a map of left fibrations $g: L \to M$ the following are equivalent:
      \begin{enumerate}
    	\item[(I)] $g: L \to M$ is a Reedy equivalence.
    	\item[(II)] $g_0: L_0 \to M_0$ is a Kan equivalence.
    	\item[(III)] For every $x: F(0) \to X$, $F(0) \times_{X} Y \to F(0) \times_{X} Z$ is a diagonal equivalence of simplicial spaces.
    \end{enumerate}
 	\item \label{leftitem:recognition} \cite[Theorem 4.39]{rasekh2017left} A map $f$ is a covariant equivalence if and only if for every map $x:F(0) \to X$, if the diagonal of the induced map 
 	$$Y \underset{X}{\times} R_x \to Z \underset{X}{\times} R_x$$
 	is a Kan equivalence. Here $R_x$ is the right fibrant replacement of the map $x$ over $X$.
	\item \label{leftitem:cov to diag} \cite[Theorem 3.17]{rasekh2017left} The following adjunction 
	\begin{center}
		\adjun{(\ss_{/X})^{Cov}}{(\ss_{/X})^{Diag}}{id}{id}
	\end{center}
	is a Quillen adjunction, which is a Quillen equivalence if $X$ is homotopically constant.
	Here the left hand side has the covariant model structure and the right hand side has the induced diagonal model structure.
 	\item \label{leftitem:exponent} \cite[Theorem 4.28]{rasekh2017left} Let $p:R \to X$ be a right fibration. The following is a Quillen adjunction:
 	\begin{center}
 		\adjun{(\ss_{/X})^{Cov}}{(\ss_{/X})^{Cov}}{p_!p^*}{p_*p^*}.
 	\end{center}
 	\item \label{leftitem:pushout product} \cite[Lemma 3.25]{rasekh2017left} Let $i:A \to B$ and $j: C \to D$ be cofibrations of simplicial spaces over $X$. If $i$ or $j$ are trivial cofibrations in the covariant model structure, then $i \square j$ is a trivial cofibration as well.
 	\item \label{leftitem:CSS invariant} \cite[Theorem 5.1]{rasekh2017left} Let $f: X \to Y$ be a map of simplicial spaces. Then the adjunction
 	\begin{center}
 		\adjun{(\ss_{/X})^{Cov}}{(\ss_{/Y})^{Cov}}{f_!}{f^*}
 	\end{center}
 	is a Quillen adjunction, which is a Quillen equivalence whenever $f$ is a CSS equivalence.
 	\item \label{leftitem:CSS fib} \cite[Theorem 5.11]{rasekh2017left} The following is a Quillen adjunction
 		\begin{center}
 		\adjun{(\ss_{/X})^{CSS}}{(\ss_{/X})^{Cov}}{id}{id}
 	\end{center}
 where the left hand side has the induced CSS model structure and the right hand side has the covariant model structure.
   \item \label{leftitem:groth} \cite[Theorem 4.18]{rasekh2017left} For a small category $\C$ there are Quillen equivalences 
    \begin{center}
   	 \begin{tikzcd}[row sep=0.5in, column sep=0.9in]
   		\Fun(\C,\s^{Kan})^{proj} \arrow[r, shift left = 1.8, "\sint_\C"] & 
   		(\ss_{/N\C})^{Cov} \arrow[l, shift left=1.8, "\sH_\C", "\bot"'] \arrow[r, shift left=1.8, "\sbT_\C"] &
   		\Fun(\C,\s^{Kan})^{proj} \arrow[l, shift left=1.8, "\sbI_\C", "\bot"']
   	 \end{tikzcd}
    \end{center}
     where the middle one has the covariant model structure the other ones have the projective model structure.
   	\item \label{leftitem:smooth maps} \cite[Corollary 5.18]{rasekh2017left}
   	Let $f:X \to Y$ be a CSS equivalence and $p:L \to Y$ a left fibration over $Y$. Then the map $f^*L \to L$ is also a CSS equivalence.
 \end{enumerate}

  Left fibrations model maps into spaces.
  Our overall goal in this paper is it to generalize all aforementioned results to the level of presheaves into higher categories.
  However, before we can do so we have to expand our playing field, which leads us to the next section.

\subsection{Bisimplicial Spaces} 
  $\sss = \Fun(\Delta^{op}, \ss) $ denotes the category of bisimplicial spaces (trisimplicial sets). 
 We have the following basic notations with regard to bisimplicial spaces:
 \begin{enumerate}
  \item Denote by $F(k,n)$ the bisimplicial space defined as
  $$F(k,n)_{abc} = \Hom_{\Delta}([a],[k]) \times \Hom_{\Delta}([b],[n]).$$
  Note in particular we have bijection $F(k,n) \cong F(k,0) \times F(0,n)$.
 \item Let $\partial F(k,n) \to F(k,n)$ denote the map  
 $(\partial F(k,0) \to F(k,0)) \square (\partial F(0,n) \to F(0,n))$, which we consider the boundary of $F(k,n)$. 
  \item The category $\sss$ is enriched over spaces
  $$\Map_{\sss}(X,Y)_n = \Hom_{\sss}(X \times \Delta[n], Y).$$
  \item The category $\sss$ is also enriched over itself
  $$(Y^X)_{knl} = \Hom_{\sss}(X \times F(k,n) \times \Delta[l], Y).$$
  \item By the Yoneda lemma, for a simplicial space $X$ we have a bijection of spaces
  $$X_{kn} \cong \Map_{\ss}(F(k,n),X).$$
  \end{enumerate}

\subsection{Reedy Model Structure on Bisimplicial Spaces} \label{subsec:biReedy}
  The category of bisimplicial spaces has a Reedy model structure \cite{reedy1974modelstructure}, which is defined as follows:
 \begin{itemize}
  \item[(F)] A map $f: Y \to X$ is a (trivial) Reedy fibration if for each $k, n \geq 0$ the following map of spaces 
  $$ \Map_{\sss}(F(k,n),Y) \to \Map_{\sss}(\partial F(k,n),Y) \underset{\Map_{\sss}(\partial F(k,n), X)}{\times} \Map_{\sss}(F(k,n), X)$$
  is a (trivial) Kan fibration
  \item[(W)] A map $f:Y \to X$ is a Reedy equivalence if it is a level-wise Kan equivalence.
  \item[(C)] A map $f:Y \to X$ is a Reedy cofibration if it is an inclusion.
 \end{itemize}
 The Reedy model structure is {\it combinatorial}, {\it simplicial} and {\it proper}. 
 Moreover, it is also {\it compatible with Cartesian closure} (\cref{rem:jt calculus}).
 These properties in particular imply that we can apply Bousfield localizations to the Reedy model structure. 
 See \cite{hirschhorn2003modelcategories} for more details.
   In order to avoid confusion we will call the Reedy model structure on bisimplicial spaces, the biReedy model structure.  

 \subsection{Diagonal Reedy Model Structures}  \label{Subsec Reedy Diagonal and Reedy Model Structures}
 In \cite[Subsection 2.5]{rasekh2017left} we studied important localizations of the Reedy model
 structure on simplicial spaces that are Quillen equivalent to the Kan model structure. 
 In a similar manner, we need localizations of the biReedy model structure that are Quillen equivalent to the Reedy model structure, 
 so we will introduce them right here. We will only state the relevant notation and leave the theorems without proofs.
 
  \begin{notone} \label{not:bunch of functors simplicial}
  Let
  $$\Diag_1 : \Delta \times \Delta \to \Delta \times \Delta \times \Delta$$
  be the functor given by
  $\phiDiag([n],[l]) = ([n],[l],[l])$. 
  Similarly, for $j = 1,2$ let
  $$p_1, p_2: \Delta \times \Delta \times \Delta \to  \Delta \times \Delta$$
  be given by $p_1([k],[n],[l]) = ([n],[l])$ and $p_2([k],[n],[l]) = ([k],[l])$. 
 \end{notone}

  We want show that $\sss$ has a model structure such that $((\Diag_1)^*,(\Diag_1)_*)$ becomes a Quillen equivalence.

 \begin{theone} \label{The Diagonal Reedy Model Structure}
  There is a unique, cofibrantly generated, simplicial model structure on $\sss$, 
  called the {\it diagonal Reedy Model Structure} and denoted by $\sss^{DiagRee}$, with the following specifications.
  
  \begin{itemize}
   \item[C] A map $f:X \to Y$ is a cofibration if it is an inclusion.
   \item[W] A map $f: X \to Y$ is a weak equivalence if $(\Diag_1)^*(f): (\Diag_1)^*(X) \to (\Diag_1)^*(Y)$ is a Reedy equivalence.
   \item[F] A map $f:X \to Y$ is a fibration if it satisfies the right lifting condition for trivial cofibrations.
  \end{itemize}
  In particular, an object $W$ is fibrant if it is biReedy fibrant and the map $(p_1)_* (p_1)^*W \to W$ is a biReedy equivalences.
 \end{theone}
  
  These model structures all give us following long chain of Quillen equivalences.
  
 \begin{theone} \label{The Quillen equivalence of Reedy model structures}
  The following is a simplicially enriched Quillen equivalence
  \begin{center}
  	  \adjun{\sss^{DiagRee}}{\ss^{Ree}}{(\Diag_1)^*}{(\Diag_1)_*}.
  \end{center}
 \end{theone}
  
  The proof is analogous to the proof of \cite[Theorem 2.13]{rasekh2017left}.

\subsection{Notational Convention for Bisimplicial Functors} \label{subsec:notation bisimp}
Some functors we have defined until now will be particularly important and hence we will give them more descriptive names.
	We use the following notation for three functors $\sss \to \ss$:
	\begin{itemize}
		\item $\LFib_n = (\iF^{F(0,n)})_*$: {\it The underlying $n$-level left fibration}. In particular, if $n=0$
		we denote it by $\LFib$ and call it the {\it underlying left fibration}.
		\item $\Val_k = (\iphi^{F(k,0)})_* $: The {\it $k$-level value}. In particular, if $k=0$ we denote it by $\Val$ and call it 
		the {\it value}.
		\item $\fDiag = \phiDiag$: The {\it diagonal}.
	\end{itemize}
	On the other hand, we use the following notation for two functors $\ss \to \sss$:
	\begin{itemize}
		\item $\LEmb = (\iF)^*$: The {\it left fibration embedding}.
		\item $\VEmb = (\iphi)^*$: The {\it value embedding}.
	\end{itemize}
The terminology above is motivated by the fact that in the next section we will define a new notion of fibration $p:Y \to X$ such that $\LFib(p): \LFib(Y) \to X$ is a left fibration and $\Val(Y)$ will give us the fibers, representing the {\it values}.

\section{The Reedy Covariant Model Structure} \label{Sec The Reedy Covariant Model Structure}
In this section we generalize the covariant model structure to the category of bisimplicial spaces over a simplicial space.
This gives us a good model for maps valued in simplicial spaces and 
the room we need to further define new model structures.

 \begin{notone} \label{Rem Setting the base}
 	For the remaining sections we will fix a simplicial space $X$ and denote the bisimplicial space $\LFib(X)$simply by $X$ to simplify notation.
 \end{notone}

 \begin{defone} \label{Def Reedy left fibration}
  Let $X$ be a simplicial space. 
  We say a map of bisimplicial spaces $p:Y \to X$ is a {\it Reedy left fibration} if it is a biReedy fibration and 
  for all $k, n \geq 0$ the following is a homotopy pullback square (using the notation \setref{item:morphism notation})
  \begin{center}
   \pbsq{Y_{kn}}{Y_{k0}}{X_{n}}{X_{0}}{\ordered{0}^*}{p_{kn}}{p_{k0}}{\ordered{0}^*}.
  \end{center}
 \end{defone}
  Notice this definition is equivalent to saying that the map is a biReedy fibration and for any $k \geq 0$, $Y_k \to X$ is a left fibration.
 As in the case of left fibrations this construction comes with a model structure, the {\it Reedy covariant model structure}.
 
 \begin{theone} \label{The Reedy Covariant Model Structure}
  Let $X$ be a simplicial space considered a bisimplicial space via \cref{Rem Setting the base}. 
  There is a unique simplicial combinatorial left proper model structure on the category $\sss_{/X}$, 
  called the Reedy covariant model structure and denoted by $(\sss_{/X})^{ReeCov}$, which satisfies following conditions:
  \begin{enumerate}
   \item An object $L \to X$ is fibrant if it is a Reedy left fibration.
   \item A map is a cofibration if it is a monomorphism.
   \item A map is a weak equivalence if it is a level-wise covariant equivalence over $X$.
   \item A weak equivalence (fibration) between Reedy left fibrations is precisely level-wise Reedy equivalence (biReedy fibration).
  \end{enumerate}
 \end{theone}

 \begin{proof}
  Starting with the Reedy model structure on $(\ss_{/X})^{Ree}$, we can construct two model structures on the category $\sss_{/X}$.
  \begin{enumerate}
   \item First, we can localize the Reedy model structure with respect to map $F(0) \to F(n) \to X$ 
   to get the covariant model structure $(\ss_{/X})^{Cov}$. 
   Then we can take simplicial objects in this model structure, which gives us the category $\sss_{/X}$, and give it the Reedy model structure.  
   
   It immediately satisfies following conditions:
   \begin{itemize}
    \item It is simplicial combinatorial left proper.
    \item Cofibrations are monomorphisms.
    \item Weak equivalences are level-wise covariant equivalences over $X$.
   \end{itemize}
   \item Alternatively, we can first take the Reedy model structure on $\sss_{/X}$. Then we can localize the model structure with respect to 
   maps $F(k,0) \to F(k,n) \to X$ to get a model structure on $\sss_{/X}$ which immediately satisfies following conditions:
   \begin{itemize}
    \item It is simplicial combinatorial left proper.
    \item Cofibrations are monomorphisms
    \item The fibrant objects are Reedy left fibrations.
    \item A weak equivalence (fibration) between Reedy left fibrations is precisely level-wise Reedy equivalence (biReedy fibration).
   \end{itemize}
  \end{enumerate}
   Thus, the theorem follows if we can prove that these two model structures coincide. As we already have the same cofibrations, 
   it suffices to prove both model structures have the same fibrant objects. 
   In order to do that we need to better understand the fibrant objects in the two model structures. 
   
   Let $p:L \to X$ be a map of bisimplicial spaces. Let $M_kL$ be the corresponding matching object. 
   An object, is fibrant in the first model structure if the maps $L_k \to M_kL$ are left fibrations of simplicial spaces for all $k \geq 0$, where $M_kL$ is the {\it matching object} \cite[Section 5.2]{hovey1999modelcategories}. 
   On the other side, $p$ is fibrant in the second model structure if the maps $L_k \to X$ are left fibrations of simplicial spaces for all $k \geq 0$. 
   We need to prove these two conditions coincide. 
   
   In order to prove it we use following commutative triangle
   \begin{center}
      \begin{tikzcd}[row sep=0.15in, column sep=0.15in]
    	L_k \arrow[dr] \arrow[rr] & & M_kL \arrow[dl] \\
   	    & X & 
   	  \end{tikzcd}.
   \end{center}
	Let $p$ be fibrant in the first model structure. We want to prove that $L_k \to X$ is a left fibration. We proceed by induction. The case $k=0$ follows from the fact that $M_0L = X$. Assume that $L_0, ..., L_k$ are left fibrations over $X$. We want to prove that $L_{k+1} \to X$ is a left fibration over $X$. By construction the map $M_{k+1}L \to X$ is a limit of a diagram in $\ss_{/X}$ with value objects $L_m \to X$ (where $m \leq k$). By induction assumption these are all left fibrations and left fibrations are closed under limits and so $M_{k+1}L \to X$ is a left fibration.  
	The result now follows from the fact that left fibrations are closed under composition.
	
	On the other side assume $p$ is fibrant in the second model structure. We want to prove that $L_k \to M_k \to L$ is a left fibration. By assumption $L_k \to X$ are left fibrations for all $k$ and so $M_kL \to X$ is also a left fibration, as it is a limit with value $L_k$. 
	The result now follows from the fact that in the commutative triangle above the two legs are left fibrations. 
  \end{proof}

   The key input of the proof is that the following two different ways of constructing model structures on $\sss_{/X}$ coincide:
   \begin{center}
    \begin{tikzcd}[row sep=0.3in, column sep=0.4in]
     (\ss_{/X})^{Ree} \arrow[r, "Ree" description, rightsquigarrow] \arrow[d, "cov" description, rightsquigarrow] & 
     (\sss_{/X})^{Ree} \arrow[d, "cov" description, rightsquigarrow] \\
     (\ss_{/X})^{Cov} \arrow[r, "Ree" description, rightsquigarrow] & (\sss_{/X})^{ReeCov}
    \end{tikzcd}.
  \end{center}
    So, the Reedy covariant model structure on bisimplicial spaces over $X$ has two 
    perspectives:
    \begin{itemize}
     \item It is a Reedy model structure, which allows us to easily characterize the weak equivalences. 
     \item It is a localization model structure, which allows us to easily characterize the fibrant objects.
    \end{itemize}
   That is the reason why we were able to give such an elegant characterization of the Reedy covariant model structure. 
   We can now use this characterization to directly generalize results about left fibrations to Reedy left fibrations, using the fact that many results about a model category generalize to its Reedy model category.
  
  \begin{remone} \label{rem:reedy contravariant model structure}
   In analogy with the duality between left and right fibrations, we also have a notion of {\it Reedy right fibrations} and similarly, 
   a {\it Reedy contravariant model structure}, which can be defined and constructed similar to \cref{The Reedy Covariant Model Structure}.
   Hence we will refrain from making those similar definitions explicit.
  \end{remone}
 
 \begin{lemone} \label{lemma:local Reedy covar}
 	Let $p: Y \to X$ be a biReedy fibration over $X$.
 	The following are equivalent.
 	\begin{enumerate}
 		\item $p$ is a Reedy left fibration.
 		\item For every map $\sigma: F(n) \times \Delta[l] \to X$ the induced map 
 		$\sigma^*Y \to F(n) \times \Delta[l]$ is a Reedy left fibration.
 		\item For every map $\sigma: F(n) \to X$ the induced map 
 		$\sigma^*Y \to F(n)$ is a Reedy left fibration.
 	\end{enumerate}
 \end{lemone}

 This follows from the fact that Reedy left fibrations are determined level-wise and \leftref{leftitem:local left fib}.
 \begin{theone} \label{The Base change Reedy left}
  Let $f: X \to Y$ be map of simplicial spaces. Then the following adjunction 
  \begin{center}
   \adjun{(\sss_{/X})^{ReeCov}}{(\sss_{/Y})^{ReeCov}}{f_!}{f^*}
  \end{center}
  is a Quillen adjunction, which is a Quillen equivalence if $f$ is a CSS equivalence.
 \end{theone}
  This follows directly from applying \cite[Proposition 15.4.1]{hirschhorn2003modelcategories} to \leftref{leftitem:CSS invariant}.
  Similarly applying the same proposition to \leftref{leftitem:cov to diag} gives us the next result.

  \begin{theone} \label{The DiagRee local ReeCov}
  The following adjunction 
  \begin{center}
   \adjun{(\sss_{/X})^{ReeCov}}{(\sss_{/X})^{DiagRee}}{id}{id}
  \end{center}
  is a Quillen adjunction, which is a Quillen equivalence if $X$ is homotopically constant.
  Here the left hand side has the Reedy covariant model structure and the right hand side has the induced diagonal Reedy model structure.
  In particular, the diagonal Reedy model structure is a localization of the Reedy covariant model structure.
 \end{theone}
 
 \begin{lemone} \label{Lemma Exp of Reedy left fib is Reedy left fib}
   Let $i:A \to B$ and $j: C \to D$ be cofibrations of bisimplicial spaces over $X$. 
  If $i$ or $j$ are trivial cofibrations in the Reedy covariant model structure,
  then $i \square j$ is a trivial cofibration as well.
 \end{lemone}
  This follows from \leftref{leftitem:pushout product} and the fact that cofibrations and trivial cofibrations are determined level-wise.
 
 \begin{theone} \label{The Pullback preserves Reedy cov equiv} 
  Let $p:R \to X$ be a Reedy right fibration. The following is a Quillen adjunction
  \begin{center}
   \adjun{(\sss_{/X})^{ReeCov}}{(\sss_{/X})^{ReeCov}}{p_!p^*}{p_*p^*}.
  \end{center}
 \end{theone}
  This follows from applying \cite[Proposition 15.4.1]{hirschhorn2003modelcategories} to \leftref{leftitem:exponent}.
 
  \begin{theone} \label{The Recognition principle for Reedy covariant equivalences} 
  Let $X$ be a simplicial space considered a bisimplicial space (\cref{Rem Setting the base}). 
  Then a map of bisimplicial spaces $Y \to Z$ over $X$ is a Reedy covariant equivalence if and only if for each map $x: F(0) \to X$
  the induced map
  $$   Y \underset{X}{\times} \LFib(R_x) \to Z \underset{X}{\times} \LFib(R_x) $$
  is a diagonal Reedy equivalence. Here $R_x$ is a choice of contravariant fibrant replacement of $x$ in $\ss_{/X}$.
 \end{theone}
 This follows from \leftref{leftitem:recognition} and the fact that weak equivalences are determined level-wise.
 
 \begin{remone}
  It is interesting to compare this result to the one for simplicial spaces.
  Despite the fact that we generalized everything to the bisimplicial setting, the contravariant fibrant replacements have remained simplicial spaces.
  
  The underlying reason is that for a map $x: F(0) \to X$, contravariant fibrant replacements and 
  Reedy contravariant fibrant replacements are the same. Indeed for an arbitrary Reedy right fibration $R \to X$ we have 
  $$ \Map_{\sss_{/X}}(F(0,0),R) \xrightarrow{ \ \simeq \ } \Map_{\ss_{/X}}(F(0),R_0) 
  \xrightarrow{ \ \simeq \ } \Map_{\ss_{/X}}(R_x,R_0) \xrightarrow{ \ \simeq \ } \Map_{\sss_{/X}}(R_x,R)$$
 where we used the fact that $R_0 \to X$ is a right fibration.
 \end{remone}
 
 Similar to the case of covariant model structure, weak equivalences between fibrant objects can be characterized 
 in much easier ways applying \leftref{leftitem:equiv left fib} level-wise.
 
 \begin{theone} \label{The Equiv of Reedy left fib}
  Let $L$ and $M$ be two Reedy left fibrations over $X$.
  Let $g:L \to M$ be a map over $X$. Then the following are equivalent.
  \begin{enumerate}
   \item $g: L \to M$ is a biReedy equivalence.
   \item $\Val(g): \Val(Y) \to \Val(Z)$ is a Reedy equivalence.
   \item For every $x: F(0) \to X$, $F(0) \times_{X} Y \to F(0) \times_{X} Z$ is a diagonal Reedy equivalence of bisimplicial spaces.
  \end{enumerate}
 \end{theone}
  Finally, we can also recover the {\it Grothendieck construction}.
    Let $\ssint_\C$, $\ssH_\C$, $\ssbT_\C$ and $\ssbI_\C$ be the functors $\sint_\C$, $\sH_\C$, $\sbT_\C$ and $\sbI_\C$ 
    defined level-wise, as further explained in \leftref{leftitem:groth}. Applying \cite[Proposition 15.4.1]{hirschhorn2003modelcategories} to these adjunctions gives us following result.

  \begin{theone} \label{the:simplicial groth construction}
  Let $\C$ be a small category.
   The two simplicially enriched adjunctions  
    \begin{center}
    \begin{tikzcd}[row sep=0.5in, column sep=0.9in]
     \Fun(\C,\ss^{Ree})^{proj} \arrow[r, shift left = 1.8, "\ssint_\C"] & 
     (\sss_{/N\C})^{ReeCov} \arrow[l, shift left=1.8, "\ssH_\C", "\bot"'] \arrow[r, shift left=1.8, "\ssbT_\C"] &
     \Fun(\C,\ss^{Ree})^{proj} \arrow[l, shift left=1.8, "\ssbI_\C", "\bot"']
    \end{tikzcd}
   \end{center}
   are Quillen equivalences. Here $\Fun(\C,\ss)$ has the projective model structure and $\sss_{/N\C}$ has the Reedy covariant model structure over $N\C$.
  \end{theone}
 
 \begin{remone}
  Here we only mentioned the Grothendieck construction over nerves of categories. 
  However, we also have a Grothendieck construction over arbitrary simplicial spaces.
  Indeed, this follows from the Quillen equivalence between the covariant model structure over simplicial spaces and 
  simplicial sets (\cite[Appendix B]{rasekh2017left}) and the straightening construction for the covariant model structure
  \cite[Chapter 2]{lurie2009htt}.
 \end{remone}

\section{Localizations of Reedy Left Fibrations} \label{Sec Localizations of Reedy Right Fibrations}
 In \cref{Sec The Reedy Covariant Model Structure} we defined fibrations which we should think of as modeling functors valued in simplicial spaces (as has been illustrated in \cref{the:simplicial groth construction}).
In this section we want to study functors valued in localizations of simplicial spaces.
In the next section we will then apply these results to functors valued in Segal spaces, complete Segal spaces and homotopically constant simplicial spaces.

\begin{notone} \label{not:local}
	As this whole section is focused on the study of Bousfield localizations we will establish following terminology with 
	regard to localizations.
	\begin{itemize}
		\item Throughout $S$ will be a set of monomorphisms in the category simplicial spaces $\ss$.
		\item A simplicial space $X$ is called {\it local with respect to $S$} if for every every map 
		$f: A \to B$ in $S$, 
		$$\Map_{\ss}(B,X) \to \Map_{\ss}(A,X)$$
		is a Kan equivalence.
		\item A bisimplicial space $X$ is called {\it local with respect to $S$} if 
		$\Val(X)$ is local with respect to $S$.
		This is equivalent to 
		$$\Map_{\sss}(\VEmb(B),X) \to \Map_{\sss}(\VEmb(A),X)$$
		being a Kan equivalence for every $A \to B$ in $S$.
		\item A map of bisimplicial spaces $p:Y \to X$ is called {\it local with respect to $S$} 
		if the map 
		$$
		\Map_{\sss}(\VEmb(B),Y) \to \Map_{\sss}(\VEmb(A),Y) \times_{\Map_{\sss}(\VEmb(A),X)} \Map_{\sss}(\VEmb(B),X)
		$$
		is a weak equivalence for every map $f: A \to B$ in $S$. Note this is equivalent to the condition that
		 for every map $f: A \to B$ in $S$ and every map $\VEmb(B) \to X$, the induced map 
		$$\Map_{/X}(\VEmb(B),Y) \to \Map_{/X}(\VEmb(A),Y)$$
		is a Kan equivalence.
	\end{itemize}
\end{notone}

We can now use the intuition outlined above to give following definition. 
Here, recall, for a given simplicial space $X$, we denote the bisimplicial space $\LEmb(X)$ again by $X$, to simplify notation
(\cref{Rem Setting the base}).

\begin{defone} \label{def:localized reedy left}
	Let $X$ be a simplicial space and $S$ a set of monomorphisms of simplicial spaces (not over $X$).
	Then a Reedy left fibration $p:L \to X$ is an {\it $S$-localized Reedy left fibration} if it is local with respect to $S$.
\end{defone}

 The goal of this section is to study this fibration. In particular, we want to show:
 \begin{enumerate}
 	\item It comes as the fibrant objects of a model structure on $\sss_{/X}$ (\cref{The Localized Reedy contravariant model structure}).
 	\item We can give various alternative characterizations of the fibrant objects (\cref{the:localized left local}/\cref{cor:localized left}).
 	\item We can give a detailed characterization of the weak equivalences (\cref{the:recognition principle for localized Reedy right equivalences}).
 \end{enumerate}

\subsection{A Tale of Three Localization Model Structures}
In this subsection we want prove that $S$-localized Reedy left fibrations are fibrant objects in a model structure, 
the {\it $S$-localized Reedy covariant model structure}. 
Moreover, in order to study its fibrant objects and weak equivalences (in \cref{subsec:understanding things}), we introduce several related model structures, the {\it $S$-localized Reedy model structure} and {\it diagonal $S$-localized Reedy model structure}.
Finally, we end this subsection by proving a Grothendieck construction for $S$-localized Reedy left fibrations over nerves 
of categories. 

\begin{theone} \label{The Localized Reedy model structure} 
 Let $S$ be a set of monomorphisms of simplicial spaces.
 There is a unique, combinatorial left proper simplicial model structure on $\ss$, denoted by $\ss^{Ree_S}$ and called the 
 {\it $S$-localized Reedy model structure},
 defined as follows. 
 \begin{itemize}
  \item[C] A map $Y \to Z$ is a cofibration if it is an inclusion.
  \item[F] An object $W$ is fibrant if it is Reedy fibrant and local with respect to $S$.
  \item[W] A map $Y \to Z$ is a weak equivalence if for every fibrant object $W$ the map 
  $$\Map_{\ss}(Z, W) \to \Map_{\ss}(Y,W)$$
  is a Kan equivalence.
 \end{itemize}
\end{theone}

\begin{proof}
 This is a direct application of a left Bousfield localization to the Reedy model structure on simplicial spaces 
 \cite[Theorem 4.1.1]{hirschhorn2003modelcategories}.
\end{proof}

\begin{theone} \label{The Diagonal localized Reedy model structure}
 Let $S$ be a set of monomorphisms of simplicial spaces.
 There is a unique simplicial combinatorial left proper model structure on $\sss$, denoted by $\sss^{DiagRee_S}$ and called the {\it diagonal $S$-localized Reedy model structure},
 defined as follows. 
 \begin{enumerate}
  \item A map $Y \to Z$ is a cofibration if it is an inclusion.
  \item A map $g:Y \to Z$ is a weak equivalence if the diagonal map 
  $$\fDiag(g): \fDiag(Y) \to \fDiag(Z)$$ is an $S$-localized Reedy equivalence.
  \item An object $W$ is fibrant if and only if it is fibrant in the diagonal Reedy model structure
  and local with respect to $S$. 
  \item  The following adjunction 
  \begin{center}
  	\adjun{(\sss)^{DiagRee_S}}{\ss^{Ree_S}}{\fDiag = (\phiDiag)^*}{(\phiDiag)_*}
  \end{center}
  is a Quillen equivalence. Here the left hand side has the diagonal localized Reedy model structure and the 
  right hand side has the localized Reedy model structure.
 \end{enumerate}
\end{theone}

\begin{proof}
  By \cref{The Quillen equivalence of Reedy model structures} we have a simplicial Quillen equivalence
  \begin{center}
  	  \adjun{(\sss)^{DiagRee}}{\ss^{Ree}}{\fDiag =(\phiDiag)^*}{(\phiDiag)_*}
  \end{center}
 which gives us a simplicial Quillen adjunction with fully faithful derived right adjoint
 \begin{center}
	\adjun{(\sss)^{DiagRee}}{\ss^{Ree_S}}{\fDiag =(\phiDiag)^*}{(\phiDiag)_*}.
 \end{center}
 Hence, by \cite[Corollary A.3.7.10]{lurie2009htt} there exists a new unique simplicial, combinatorial, left proper 
 model structure on $\sss$, the {\it diagonal $S$-localized model structure}, which satisfies 
 following conditions:
 \begin{enumerate}
 	\item Cofibrations are monomorphisms.
 	\item Weak equivalences are $S$-localized diagonal equivalences.
 	\item The adjunction $(\fDiag =(\phiDiag)^*,(\phiDiag)_*)$ is a simplicial Quillen equivalence between this model structure 
 	and the $S$-localized Reedy model structure on $\ss$.
 	\item An object $X$ is fibrant if it is biReedy fibrant and biReedy equivalent to $(\phiDiag)_*(Y)$ where $Y$ is $S$-local, 
 	meaning that $X$ is fibrant in the diagonal Reedy model structure and $S$-local. \qedhere
 \end{enumerate}
\end{proof}

\begin{theone} \label{The Localized Reedy contravariant model structure} 
 Let $S$ be a set of monomorphisms of simplicial spaces.
 There is a unique simplicial combinatorial left proper model structure on $\sss_{/X}$, denoted by $(\sss_{/X})^{ReeCov_S}$ and called the 
 {\it $S$-localized Reedy covariant model structure}, defined as follows. 
 \begin{itemize}
  \item[C] A map $Y \to Z$ over $X$ is a cofibration if it is an inclusion.
  \item[F] An object $Y \to X$ is fibrant if it is a Reedy left fibration and local with respect to $S$.
  \item[W] A map $Y \to Z$ over $X$ is a weak equivalence if for every fibrant object $W \to X$ the map 
  $$\Map_{/X}(Z, W) \to \Map_{/X}(Y,W)$$
  is a Kan equivalence.
 \end{itemize}
\end{theone}

\begin{proof}
	Notice the model structure on the category $\sss_{/X}$ is still proper and cellular \cite[Proposition 12.1.6]{hirschhorn2003modelcategories} 
	and so we can apply left Bousfield localization	\cite[Theorem 4.1.1]{hirschhorn2003modelcategories} with respect to the set of morphisms
    $\mathcal{L} = \{ \VEmb(A) \to \VEmb(B) \to X : A \to B \in S \}.$
\end{proof}
 
 Combining \cref{The DiagRee local ReeCov} with \cref{the:localization of Quillen equiv}
 gives us following similar result.
 
 \begin{propone} \label{propone:local relative gives local}
  The following adjunction
  \begin{center}
   \adjun{(\sss_{/X})^{ReeCov_S}}{(\sss_{/X})^{DiagRee_S}}{id}{id}
  \end{center}
  is a Quillen adjunction, which is a Quillen equivalence whenever $X$ is homotopically constant. 
  Here the left hand side has the localized Reedy covariant model structure and the left hand side 
  has the induced diagonal localized Reedy model structure over the base $X$.
 \end{propone}

 One very important instance is the case $X=F(0)$. The theorem shows that $\sss^{ReeCov_S}$ is the same as $\sss^{DiagRee_S}$.

We move on to prove the Grothendieck construction for localized Reedy left fibrations over nerves of categories.
Before that let us recall that an object in the projective model structure on $\Fun(\C,\ss^{Ree_S})^{proj}$ (where $\ss$ has the $S$-localized Reedy model structure) is fibrant if it is fibrant in the projective model structure on $\Fun(\C,\ss^{Ree})^{proj}$ (where now $\ss$ has the Reedy model structure) and is local with respect to natural transformations 
$$\id \times f: \Hom_\C(c,-) \times A \to \Hom_\C(c,-) \times B$$
for all objects $c$ and maps $f: A \to B$ in $S$. 

\begin{theone} \label{the:simplicial groth construction localized}
	Let $\C$ be a small category. Then the adjunctions defined in \cref{the:simplicial groth construction}
	 \begin{center}
		\begin{tikzcd}[row sep=0.5in, column sep=0.9in]
			\Fun(\C,\ss^{Ree_S})^{proj} \arrow[r, shift left = 1.8, "\ssint_\C"] & 
			(\sss_{/N\C})^{ReeCov_S} \arrow[l, shift left=1.8, "\ssH_\C", "\bot"'] \arrow[r, shift left=1.8, "\ssbT_\C"] &
			\Fun(\C,\ss^{Ree_S})^{proj} \arrow[l, shift left=1.8, "\ssbI_\C", "\bot"']
		\end{tikzcd}
	\end{center}
 are simplicial Quillen equivalences. Here the middle has the $S$-localized Reedy covariant model structure
 and the two sides have the projective model structure on the $S$-localized Reedy model structure. 
\end{theone}

\begin{proof}
    First we show both are Quillen adjunctions. As both left adjoints still preserve cofibrations by \cite[Corollary A.3.7.2]{lurie2009htt} it suffices to prove that the right adjoints preserve fibrant objects. By \cref{the:simplicial groth construction}, the right adjoints preserve fibrant objects in the unlocalized model structures, so we only need to confirm that they preserve local objects.
    
    Using the adjunctions this is equivalent to proving the following statements: 
    $$\ssint_\C \Hom_\C(c,-) \times A \to \ssint_\C \Hom_\C(c,-) \times B$$
    is an $S$-localized Reedy covariant equivalence over $\C$ and
    $$\ssT_\C(\VEmb(A)) \to \ssT_\C(\VEmb(B))$$
    is an $S$-localized projective equivalence. 
    Both of those are immediate computations.
    
	We now move on to proving they are Quillen equivalences.
	The composition map $\ssbT_\C \circ \ssint_\C$ is naturally equivalent to the identity functor and so is a Quillen equivalence.
	Hence it suffices to prove that the adjunction $(\sint_\C,\ssH_\C)$ is a Quillen equivalence. 
	
	By \cref{the:simplicial groth construction}, the counit map is an equivalence. 
	So, we move on to the derived unit map. Again, by \cref{the:simplicial groth construction}, for a fibrant object 
	$F: \C \to \ss$, the map $\ssint_\C F \to \ssbI_\C F$ is a biReedy equivalence and hence a localized 
	Reedy covariant equivalence. Hence $\ssbI_\C F$ is a fibrant replacement of $\ssint_\C F$.
	Thus, the derived unit map is given by 
	$$F \to \ssH_\C \ssbI_\C F,$$
	which is indeed an equivalence as it is naturally equivalent to the identity as explained above. 
\end{proof}
 
 The result has several important corollaries that we will use in the next subsection.
 
 \begin{corone}
 	Let $p:L \to N\C$ be a Reedy left fibration. Then $p$ is an $S$-localized Reedy left fibration if and only if 
 	it is fiberwise diagonal $S$-localized Reedy fibrant.
 \end{corone}

\begin{corone}
	A map of bisimplicial spaces $Y \to Z$ over $N\C$ is $S$-localized Reedy covariant equivalence if and only if 
	the map 
	$$Y \times_{N\C} N\C_{/c} \to Z \times_{N\C} N\C_{/c} $$
	is a diagonal $S$-localized Reedy equivalence. 
\end{corone}

If we let $\C = [n]$, then $N[n] = F(n)$ and so we can use the results above immediately to understand 
the $S$-localized Reedy covariant model structure over $F(n)$. 
For many applications, however, this is not good enough. We want to understand $f$-localized 
Reedy left fibrations over $F(n) \times \Delta[l]$. For that we have following result:

\begin{corone}
	The Reedy equivalence $\pi_1: F(n) \times \Delta[l] \to F(n)$ induces a Quillen equivalence 
	\begin{center}
		\adjun{(\sss_{/F(n) \times \Delta[l]})^{ReeCov_S}}{(\sss_{/F(n)})^{ReeCov_S}}{(\pi_1)_!}{(\pi_1)^*}
	\end{center}
   and so, in particular, every $S$-localized Reedy left fibration is biReedy equivalent to a map of the form 
   $p \times \Delta[l]: \ssint_\C G \times \Delta[l] \to F(n) \times \Delta[l]$, where $G: \C \to \ss$ 
   is a fibrant object in the projective model structure.
\end{corone}

We will use the local results in the next subsection to study $S$-localized Reedy left fibrations and their equivalences 
over arbitrary simplicial spaces. 

\subsection{Understanding the Localized Reedy Covariant Model Structure} \label{subsec:understanding things}
In this subsection we want to study the fibrant objects and weak equivalences in the $S$-localized 
Reedy covariant model structure over an arbitrary simplicial space $X$.

Some results will require some conditions on the set of maps $S$, which we will fix now.

\begin{notone} \label{not:conditions on f}
	Let $S$ be a set of monomorphisms of simplicial spaces.
	\begin{itemize}
		\item {\bf (S)}: A map $f$ in $S$ {\it satisfies condition} {\bf (S)} if every homotopically constant 
		simplicial space is local with respect to $f$, which is equivalent to $f$ being an equivalence in the 
		diagonal model structure.  
		The set of maps $S$ {\it satisfies condition} {\bf (S)} if 
		every map in $S$ satisfies condition {\bf (S)}.
		
		\item {\bf (D)}: A map $f: A \to B$  in $S$ {\it satisfies condition} {\bf (D)} if $B$ is diagonally contractible.
		The set of maps $S$ {\it satisfies condition} {\bf (D)} if every map in $S$ satisfies condition {\bf (D)}.
		\item {\bf (C)}: A map $f:A \to B$ {\it  satisfies condition} {\bf (C)} if it satisfies condition {\bf (S)} and {\bf (D)}.
		This an be stated directly as $A$ and $B$ being diagonally contractible.
	    The set of maps $S$ {\it satisfies condition} {\bf (C)} if 
		every map in $S$ satisfies condition {\bf (C)}.
		\item {\bf (P)}: The set of monomorphisms $S$ {\it satisfies condition} {\bf (P)} if $W$ being local with respect to $S$ 
		implies that $W^X$ is local with respect to $S$ for all simplicial spaces $X$.
	\end{itemize}
\end{notone}

\begin{exone} \label{ex:maps satisfy conditions}
	Let us see some examples of maps that satisfy these conditions:
	\begin{enumerate}
		\item The simplicial space $F(n)$ is a diagonally contractible. 
		Hence any map $A \to F(n)$ satisfies condition {\bf (D)}.
		\item The simplicial space $G(n)$ \cite[Section 5]{rezk2001css} is also diagonally contractible. Hence the inclusions $G(n) \to F(n)$ satisfy 
		condition {\bf (C)}.
	    \item The map also satisfies condition {\bf (P)} \cite[Lemma 10.3]{rezk2001css}.
	    \item Let $\C$ be a contractible category. Then any map $F(0) \to N\C$ satisfies condition {\bf (C)}.
	    \item In particular, the map $F(0) \to E(1)$ (\spaceref{item:En}) satisfies condition {\bf (C)}, but also condition {\bf (P)} \cite[Proposition 12.1]{rezk2001css}.
	\end{enumerate}
\end{exone}

We will start with characterizations of $S$-localized Reedy left fibrations. First two lemmas.

\begin{lemone} \label{lemma:localized left local}
  Let $p: L \to X$ be a biReedy fibration.
  Then the following are equivalent:
  \begin{enumerate}
   	\item For every map $\sigma: F(n) \times \Delta[l] \to X$, the pullback map $\sigma^*p: \sigma^*L \to F(n) \times \Delta[l]$
   	is an $S$-localized Reedy left fibration.
    \item For every map $\sigma: F(n) \to X$, the pullback map $\sigma^*p: \sigma^*L \to F(n)$
    is an $S$-localized Reedy left fibration.
    \item $p$ is a Reedy left fibration and for every point $\{x\}: F(0) \to X$ the fiber 
    $\Fib_xL$ is fibrant in the diagonal $S$-localized Reedy model structure.
    \item $p$ is a Reedy left fibration and for every point $\{x\}: F(0) \to X$ the fiber 
    $\Val(\Fib_xL)$ is fibrant in the $S$-localized Reedy model structure.
  \end{enumerate}
\end{lemone}

\begin{proof}
	All four statements break down into two parts: proving $p$ is a Reedy left fibration and proving it is local with respect to $S$.
	The first always follows either by definition or from \cref{lemma:local Reedy covar}. Hence we will only focus on proving it is 
	local with respect to $S$.
	
	{\it (1) $\Leftrightarrow$ (2)}
	This follows from the fact that $\pi_2: F(n) \times \Delta[l] \to F(n)$ is a Reedy equivalence and being local with respect to $S$ is invariant under Reedy equivalences.
	
	{\it (2) $\Leftrightarrow$ (3)}
	One side is immediate, for the other we will use \cref{the:simplicial groth construction localized}.
	Fix a map $\sigma: F(n) \to X$. We want to prove $\sigma^*p: \sigma^*L \to F(n)$ is an $S$-localized Reedy left fibration. 
	By assumption $p: L \to X$ is already a Reedy left fibration, which, by \cref{lemma:local Reedy covar}, implies that $\sigma^*p$ is also 
	a Reedy left fibration. 
	
	Hence, by \cref{the:simplicial groth construction}, $\sigma^*p: \sigma^*L \to F(n)$ is Reedy equivalent to a map 
	$\ssint_{[n]} G \to F(n)$, where $G: [n] \to \ss$ and as the property of being local with respect to $S$ is invariant under Reedy equivalences, $\sigma^*p$ is an $S$-localized Reedy left fibration 
	if and only if $\ssint_{[n]} G$ is local with respect to $S$. 
	By \cref{the:simplicial groth construction localized}, this itself is equivalent to 
	$G$ being fibrant in the projective model structure, which by definition means that 
    for all $0 \leq i \leq n$, $G(i)$ is fibrant in the 
	$S$-localized Reedy model structure. 
	This is directly equivalent to $\sigma^*p$ being fiber-wise fibrant in the diagonal $S$-localized Reedy model structure.
	
	{\it (3) $\Leftrightarrow$ (4)}
	This follows from the definition of fibrant objects in the diagonal $S$-localized model structure on $\sss$. 
\end{proof}

\begin{lemone} \label{lemma:localized left global}
	Assume that $S$ satisfies condition {\bf (S)} and $p: L \to S$ is a biReedy fibration.
	Then the following are equivalent.
	\begin{enumerate}
		\item $p$ is an $S$-localized Reedy left fibration.
		\item $p$ is a Reedy left fibration and the simplicial space $\Val(L)$ is local with respect to $S$.
		\item $p$ is a Reedy left fibration and the simplicial spaces $\Val_k(L)$ are local with respect to $S$ for all $k \geq 0$.
	\end{enumerate}
\end{lemone}

\begin{proof}
	{\it (1) $\Leftrightarrow$ (2)}
  	Let $p$ be an $S$-localized Reedy left fibration. 
  We have a commutative diagram 
  \begin{center}
  	\comsq{\Map_{/X}(\VEmb(B),L) }{\Map_{/X}(\VEmb(A),L) }{\Map_{/\Val(X)}(B,\Val(L)) }{\Map_{/\Val(X)}(A,\Val(L))}{}{\cong}{\cong}{}.
  \end{center}
  The vertical maps are bijections using the enriched adjunction $(\VEmb,\Val)$.
  So the top map is an equivalence (which is the 
  definition of a localized Reedy left fibration) if and only if the bottom map is an equivalence (which is equivalent to 
  $\Val(L) \to \Val(X)$ being fibrant in the localized model structure on $\ss_{/\Val(X)}$).
  
  As $S$ satisfies condition {\bf (S)}, $\Val(X) = X_0$ is local with respect to $S$ and so the bottom map being an equivalence is equivalent to 
  $\Val(L)$ being local with respect to $S$, finishing the proof.
  
  {\it (2) $\Leftrightarrow$ (3)}
  One side side is a special case. For the other side, notice we have a Reedy equivalence of simplicial spaces 
  $$\Val_n(L) \simeq \Val(L) \times_{X_n} X_0.$$
  The right hand side is local with respect to $S$ ($\Val(L)$ by assumption and $X_n$, $X_0$ by condition {\bf (S)}), 
  hence the right hand is local as well.
\end{proof}

\begin{theone} \label{the:localized left local}
	If $p$ is an $S$-localized Reedy left fibration, then it satisfies the conditions of \cref{lemma:localized left local}. The opposite holds if $S$ satisfies condition {\bf (D)}. 
\end{theone}

\begin{proof}
	If $p$ is an $S$-localized Reedy left fibration, then it satisfies Condition $(1)$ of \cref{lemma:localized left local}, as fibrations are closed under pullback. On the other side, assume $S$ satisfies condition {\bf (D)} and assume $p$ is a biReedy fibration. We will prove that Condition $(2)$ of \cref{lemma:localized left local} implies $p$ is an $S$-localized Reedy left fibration.

	By \cref{lemma:local Reedy covar}, Condition (2) implies that $p$ is a Reedy left fibration, so we only need to show it is local with respect to $S$. 
	It suffices to prove that $p$ satisfies the right lifting property with respect to the cofibration $j$ defined as the pushout product 
	$$j=(\VEmb(f): \VEmb(A) \to \VEmb(B)) \square (\partial \Delta[n] \to \Delta[n]),$$ 
	where $f$ is in $S$. The codomain of $j$ is $\VEmb(B) \times \Delta[n]$ and every map $\sigma:\VEmb(B) \times \Delta[n] \to X$ factors through a map $\delta:\Delta[n] \to X$ (as $S$ satisfies condition {\bf (D)}. Hence we get following diagram 
	\begin{center}
		\begin{tikzcd}[row sep=0.3in, column sep=0.3in]
			\ds\VEmb(A) \times \Delta[n] \coprod_{\VEmb(A) \times \partial \Delta[n]} \VEmb(B) \times \partial \Delta[n] \arrow[r] \arrow[d, "j"] & \delta^*L \arrow[r] \arrow[d, "\delta^*p"] & L \arrow[d, "p"] \\
			\VEmb(B) \times \Delta[n] \arrow[r] \arrow[rr, bend right=20, "\sigma"] \arrow[ur, dashed] & \Delta[n] \arrow[r, "\delta"] & X
		\end{tikzcd}.	
	\end{center} 
	By assumption $\delta^*L \to \Delta[n]$ is an $S$-localized Reedy left fibration and so has a lift, which implies that our original lifting problem has a solution proving that $p$ is an $S$-localized Reedy left fibration.
\end{proof}

Combining \cref{the:localized left local} with \cref{lemma:localized left global} immediately gives us following result.

\begin{corone} \label{cor:localized left}
	Let $S$ satisfy condition {\bf (C)}. Then all conditions in \cref{lemma:localized left global} and \cref{lemma:localized left local} coincide.
\end{corone}

We move on to characterize weak equivalences in the $S$-localized Reedy covariant model structure.
First, observe that we have a very immediate result for weak equivalences between fibrant objects.

 \begin{theone} \label{theone:local equiv between fibrants}
	Let $L$ and $M$ be two $S$-localized Reedy left fibrations over $X$.
	Let $g:L \to M$ be a map over $X$. Then the following are equivalent.
	\begin{enumerate}
		\item $g: L \to M$ is a biReedy equivalence.
		\item $\Val(g): \Val(L) \to \Val(M)$ is a Reedy equivalence.
		\item For every $\{ x \}: F(0) \to X$, the map  $\Fib_x \Val(L) \to \Fib_x \Val(M)$ is a Reedy equivalence
		of bisimplicial spaces.
		\item For every $\{x \}: F(0) \to X$, the map $\Fib_x(L) \to \Fib_x (M)$ is a diagonal Reedy equivalence
		of bisimplicial spaces.
	\end{enumerate}
\end{theone}

Before moving to the general case, we prove a {\it recognition principle for $S$-localized Reedy covariant equivalences
between Reedy left fibrations}.

For the next proposition we need following construction.
Let $p: L \to X$ be a Reedy left fibration. Then we have following diagram
\begin{equation} \label{eq:fib rep}
	\begin{tikzcd}[row sep=0.2in, column sep=0.3in]
		L_\bullet \arrow[drr, "p"'] \arrow[rr, hookrightarrow, "i", "\simeq"'] & & 
		\tilde{L}_\bullet \arrow[d, "\tilde{p}" twoheadrightarrow] \arrow[rr, "j", "\simeq"'] 
		& & \hat{L} \arrow[dll, twoheadrightarrow, "\hat{p}"] \\
		& & X & &  	
	\end{tikzcd}.		 
\end{equation}
Here the first map is the level-wise functorial factorization of the simplicial object in $(\ss_{/X})^{Ree_S}$ in the $S$-localized Reedy model structure. Moreover, let $\hat{p}: \hat{L} \to X$ be the biReedy fibrant replacement over $X$.

\begin{propone} \label{prop:localized equiv between Reedy left}
	Let $S$ satisfy condition {\bf (C)}.
	Let $p: L \to X$, $q: M \to X$ be Reedy left fibrations (not necessarily localized) and
	let $f:L \to M$ be a map over $X$. 
	Then the following are equivalent:
	\begin{enumerate}
		\item $f$ is an $S$-localized Reedy covariant equivalence.
		\item The map $\hat{f}: \hat{L} \to \hat{M}$ constructed in \ref{eq:fib rep} is a biReedy equivalence.
		\item The map 
		$$\Val(\hat{f}): \Val(\hat{L}) \to \Val(\hat{M})$$ 
		constructed in \ref{eq:fib rep} is a Reedy equivalence.
		\item The map
		$$\Val(f): \Val(L) \to \Val(M)$$
		is an $S$-localized Reedy equivalence.
		\item For every object $\{x\}: F(0) \to X$, the induced map on fibers 
		$$\Val(\Fib_x L) \to \Val(\Fib_x M)$$
		is an $S$-localized Reedy equivalence.
		\item For every object $\{x\}: F(0) \to X$, the induced map on fibers 
		$$\Fib_x L \to \Fib_x M$$
		is a diagonal $S$-localized Reedy equivalence.
	\end{enumerate}
\end{propone}

\begin{proof} 
	{\it (1) $\Leftrightarrow$ (2)}
	It suffices to prove that the map
    $\hat{p}: \hat{L} \to X$ from \ref{eq:fib rep} is a fibrant replacement of $L \to X$ in the $S$-localized 
    Reedy covariant model structure. 
    
    For that we need to prove two statements:
    \begin{itemize}
    	\item $\hat{p}: \hat{L} \to X$ is an $S$-localized Reedy left fibration: Indeed it is biReedy fibrant by definition.
    	Moreover, 
    	$\Val(\hat{L})$ is Reedy equivalent to $\Val(\tilde{L})$, which is by definition fibrant in the $S$-localized Reedy model structure, 
    	and so is itself fibrant in the $S$-localized Reedy model structure.
    	Finally, $\hat{L} \to X$ is Reedy left fibration, as it is biReedy equivalent to $\tilde{L} \to X$ and 
    	for every $n \geq 0$ we have the commutative diagram 
    	\begin{center}
    		\begin{tikzcd}[row sep=0.3in, column sep=0.6in]
    			\Val_n(L) \arrow[r, "\Val_n(i)", "\simeq"'] \arrow[d, "\simeq"] & \Val_n(\hat{L}) \arrow[d] \\
    			\Val(L) \times_{X_0} X_n \arrow[r, "\Val(i) \times_{X_0} X_n", "\simeq"'] & \Val(\hat{L}) \times_{X_0} X_n 
    		\end{tikzcd}.
    	\end{center}
    	\item $j \circ i$ is an $S$-localized Reedy covariant equivalence. Indeed $i$ is a level-wise $S$-localized Reedy equivalence and so an equivalence in the $S$-localized Reedy covariant model structure and $j$ is a biReedy equivalence. 
    \end{itemize}
   
   Now that we have established that $\hat{L}$ is the $S$-localized Reedy covariant fibrant replacement of $L$ over $X$, it follows by 
   definition of Bousfield localizations that $f$ is an $S$-localized Reedy covariant equivalence if and only if 
   $\hat{f}$ is a biReedy equivalence. 
   
   {\it (2) $\Leftrightarrow$ (3)} 
   $\hat{L}$ and $\hat{M}$ are Reedy left fibrations and so, by \cref{The Equiv of Reedy left fib}, a map $\hat{f}: \hat{L} \to \hat{M}$ is a biReedy equivalence 
   if and only if $\Val(\hat{f}): \Val(\hat{L}) \to \Val(\hat{M})$ is a Reedy equivalence. 
   
   {\it (3) $\Leftrightarrow$ (4)}
   We have a commutative diagram 
   \begin{center}
   	\begin{tikzcd}[row sep=0.2in, column sep=0.2in]
   		\Val(L) \arrow[r] \arrow[d, "\Val(f)"'] & \Val(\hat{L}) \arrow[d, "\Val(\hat{f})"] \\
   		\Val(M) \arrow[r] & \Val(\hat{M})
   	\end{tikzcd}.
   \end{center}   
   By construction, the horizontal maps are fibrant replacements in the $S$-localized Reedy model structure. 
   Hence, $\Val(f)$ is an $S$-localized Reedy weak equivalence if an only if $\Val(\hat{f})$ is a Reedy equivalence.
   
   {\it (3) $\Leftrightarrow$ (5)}
   First, observe that $\Val(\hat{L}) \to \Val(\hat{M})$ is a Reedy equivalence if and only if for every $\{x\}: \Delta[0] \to X_0$, the 
   induced map 
   $$\Fib_x(\Val(\hat{L})) \to \Fib_x(\Val(\hat{M}))$$
   is a Reedy equivalence. 
   
   Now, for a given point $\{x\}: \Delta[0] \to X_0$. The induced map on fibers 
   $$\Fib_x(\Val(L)) \to \Fib_x(\Val(\hat{L}))$$ 
   is still the fibrant replacement in the $S$-localized Reedy model structure. 
   Hence this is equivalent to 
   $$\Fib_x(\Val(L)) \to \Fib_x(\Val(M))$$
   being a $S$-localized Reedy equivalence.
   
   {\it (5) $\Leftrightarrow$ (6)}
   This follows from the fact that $L \to X$ is a Reedy left fibration and so 
   $\VEmb\Val\Fib_x(L) \to \Fib_x(L)$ is a biReedy equivalence. 
\end{proof}

 \begin{theone} \label{the:recognition principle for localized Reedy right equivalences}
	A map $g:Y \to Z$ of bisimplicial spaces over $X$ is an equivalence 
	in the localized Reedy covariant model structure if and only if for each map $\{x\}: F(0) \to X$, the induced map 
	$$ Y \underset{X}{\times} \LFib(R_x) \to Z \underset{X}{\times} \LFib(R_x)$$
	is an equivalence in the diagonal localized Reedy model structure.
	Here $R_x$ is a choice of right fibrant replacement of the map $\{x\}$.
\end{theone}

\begin{proof}
	Let $\hat{g}: \hat{Y} \to \hat{Z}$ be a fibrant replacement of $g$ in the Reedy covariant model structure (note: not localized). 
	Moreover, let $\{x \}: F(0) \to X$ be a vertex in $X$.
	This gives us following zig-zag of maps:
	\begin{center}
		\begin{tikzcd}[row sep=0.25in, column sep=0.25in]
			\hat{Y} \underset{X}{\times} F(0) \arrow[d, "ReeContra \simeq"'] \arrow[r] & \hat{Z} \underset{X}{\times} F(0) \arrow[d, "ReeContra \simeq"] \\
			\hat{Y} \underset{X}{\times} R_x \arrow[r] &\hat{Z} \underset{X}{\times} R_x \\
			Y \underset{X}{\times} R_x \arrow[r] \arrow[u, "ReeCov \simeq"] & Z \underset{X}{\times} R_x \arrow[u, "ReeCov \simeq"']
		\end{tikzcd}
		.
	\end{center}
	According to \cref{The Pullback preserves Reedy cov equiv} the top vertical maps are Reedy contravariant equivalences 
	and the bottom vertical maps are Reedy 
	covariant equivalences. By \cref{The DiagRee local ReeCov} both of these are diagonal Reedy equivalences, which are always
	diagonal localized Reedy equivalences (\cref{The Diagonal localized Reedy model structure}). Thus the top map is a diagonal localized Reedy equivalence if and only if the bottom map is one, but by \cref{prop:localized equiv between Reedy left}	this is equivalent to $Y \to Z$ being a localized Reedy contravariant equivalence over $X$.
\end{proof}

  \begin{theone} \label{the:loc Ree contra invariant under CSS}
 	Let $g: X \to Y$ be a map of simplicial spaces. Then the adjunction
 	\begin{center}
 		\adjun{(\sss_{/X})^{ReeCov_S}}{(\sss_{/Y})^{ReeCov_S}}{g_!}{g^*}
 	\end{center}
 	is a Quillen adjunction, which is a Quillen equivalence whenever $g$ is a CSS equivalence.
 	Here both sides have the $S$-localized Reedy covariant model structure.
 \end{theone}
 
 \begin{proof}
 	Clearly it is a Quillen adjunction as fibrations are stable under pullback.
 	
 	Let us now assume that $g$ is a CSS equivalence. We want to prove that $(g_!,g^*)$ is a Quillen equivalence of $S$-localized Reedy covariant model structures. By \cref{The Base change Reedy left} it is a Quillen equivalence of Reedy covariant model structures and we want to use \cref{the:localization of Quillen equiv} to finish the proof. Unfortunately we cannot apply it directly as we have not characterized the $S$-localized Reedy covariant model structure on $\sss_{/Y}$ via $g_!$. Hence, we will prove that in this case they coincide. 
 	
 	We need to prove the following fact: Let $p:L \to Y$ be a Reedy left fibration. Then $p$ is $S$-localized if and only and only if $g^*p: g^*L \to X$ is an $S$-localized Reedy left fibration. By \leftref{leftitem:smooth maps}, $g^*L \to L$ is a level-wise CSS equivalence, which means it is a level-wise covariant equivalence (\leftref{leftitem:CSS fib}). 
 	Hence, if $g^*p$ is $S$-localized then $p$ is $S$-localized as well.
\end{proof}

 \begin{theone} \label{the:pulling back along loc Ree right preserves loc Ree cov}
 	Let $S$ satisfy conditions {\bf (P)} and {\bf (C)}.
	Let $p:R \to X$ be a Reedy right fibration over $X$. The induced adjunction
	\begin{center}
		\adjun{(\sss_{/X})^{ReeCov_S}}{(\sss_{/X})^{ReeCov_S}}{p_!p^*}{p_*p^*}
	\end{center}
	is a simplicial Quillen adjunction. Here both sides have the $S$-localized Reedy covariant model structure.
\end{theone}

\begin{proof}
	Clearly the left adjoint preserves cofibrations and so by \cite[Corollary A.3.7.2]{lurie2009htt} it suffices to show that the right adjoint preserves fibrant objects. So, let $L \to X$ be a localized Reedy left fibration over $X$. Then we have to show that $p_*p^*L \to X$ is also a localized 
	Reedy left fibration over $X$. By \cref{The Pullback preserves Reedy cov equiv}, 
	we already know that it is a Reedy left fibration, so all that is left is to show that 
	it is local with respect to $S$. 
	By \cref{def:localized reedy left}, it suffices to show that for any map 
	$q:\VEmb(B) \to X$ the induced map 
	$$ \Map_{/X}(\VEmb(B), p_*p^*L) \to \Map_{/X}(\VEmb(A), p_*p^*L)$$
	is a Kan equivalence. Using the adjunction, this is equivalent to
	$$ \Map_{/X}(p_!p^*\VEmb(B), L) \to \Map_{/X}(p_!p^*\VEmb(A),L)$$
	being a Kan equivalence. For that it suffices to show that 
	$$p_!p^*\VEmb(A) \to p_!p^*\VEmb(B)$$
	is a localized Reedy covariant equivalence over $X$. 
	
	As $S$ satisfies condition {\bf (C)}, $\fDiag(B)$ is contractible and so the map $\VEmb(B) \to X$ 
	is Reedy equivalent to a map of the form $\VEmb(B) \to F(0) \to X$. Thus 
	$$p^*(\VEmb(B)) = \VEmb(B) \times_X R \simeq \VEmb(B) \times (F(0) \times_X R)$$
	similarly $p^*(\VEmb(A)) \simeq \VEmb(A) \times (F(0) \times_X R)$.
	However, 
	$$\VEmb(A) \times (F(0) \underset{X}{\times} R) \to  \VEmb(B) \times (F(0) \underset{X}{\times} R)$$
	is a localized Reedy covariant equivalence over $X$. 
	Indeed, this immediately follows from the fact that $S$ satisfies condition {\bf (P)}.
\end{proof}

Using condition {\bf (P)} we can recover other interesting results about $S$-localized Reedy left fibrations.

\begin{propone}
	Let $S$ be a set of cofibrations that satisfy condition {\bf (P)}.
	Let $g: C \to D$ be a cofibration of bisimplicial spaces and $p: L \to X$ a $S$-localized Reedy left fibration. 
	Then $\exp{g}{p}$ is also a localized Reedy left fibration.
\end{propone}

\begin{proof}
	By \leftref{leftitem:exponent} $\exp{g}{p}$ is a Reedy left fibration and so it suffices to prove that it is local with respect to $S$.
	It suffices to observe that 
	$\exp{f}{\exp{g}{p}}$ 
	is a trivial biReedy fibration for every $f$ in $S$. By direct computation we have 
	$$ \exp{f}{\exp{g}{p}} \cong \exp{f \square g}{p} \cong \exp{g}{\exp{f}{p}}.$$
	The result now follows from the fact that $\exp{f}{p}$ is a trivial biReedy fibration (as $f$ satisfies {\bf (P)}) 
	and the biReedy model structure is compatible with Cartesian closure (\cref{subsec:biReedy}).
\end{proof}

\begin{corone} \label{cor:Exp of loc Ree right is loc Ree right}
	Let $S$ be a set of cofibrations that satisfy condition {\bf (P)}.
	Let $L \to X$ be an $S$-localized Reedy left fibration. Then for any bisimplicial space $Y$, $L^Y \to X^Y$ is also an 
	$S$-localized Reedy left fibration.
\end{corone}

\section{(Segal) Cartesian Fibrations} \label{sec:cartesian Fibrations}
 In the previous we section defined and studied fibrations with fiber localizations of Reedy fibrant simplicial spaces. 
 In this section we want to apply these results to three very important cases: Segal spaces, complete Segal spaces and homotopically constant simplicial spaces.
 Similar to the previous sections $X$ is a fixed simplicial space considered a bisimplicial space (\cref{Rem Setting the base}).
 Also recall the notion of being local as described in \cref{not:local}.
 
  \begin{defone} \label{def:segcartfib}
   We say a Reedy left (right) fibration $Y \to X$ over $X$ is a {\it Segal coCartesian fibration}
   ({\it Segal Cartesian fibration}) 
   if it is local with respect to the set of maps 
   $$S = \{ G(n) \to F(n): n \geq 2\}.$$
  \end{defone}

 \begin{defone} \label{def:cart fib}
   We say a Reedy left (right) fibration $Y \to X$ is a {\it coCartesian fibration} ({\it Cartesian fibration}) 
   if it is local with respect to the set of maps 
   $$S = \{ G(n) \to F(n): n \geq 2\} \cup \{ F(0) \to E(1) \}.$$
  \end{defone}
  
  \begin{defone} \label{def:rfib bi spaces}
   We say a Reedy left (right) fibration $Y \to X$ is a {\it left fibration} ({\it right fibration}) if it is 
   local with respect to the set of maps 
   $$S = \{ F(0) \to F(n): n \geq 0\}.$$
 \end{defone}

 By \cref{ex:maps satisfy conditions} all maps in the set 
 $$\{ G(n) \to F(n) : n \geq 2\} \cup \{ F(0) \to E(1) \} \cup \{ F(0) \to F(n) : n \geq 0 \}$$
 satisfy conditions {\bf (C)} and {\bf (P)} and so all results in \cref{Sec Localizations of Reedy Right Fibrations} hold for their corresponding fibrations. In order to summarize the results about the various localizations using the following table.
 
 \begin{center}
  \begin{tabular}{l|l|l|l|l}
   Variance $(\R)$ & Value $(\V)$& Fibration $(\F)$ & Model Structure $(\M)$ & Denoted $(\D)$ \\ \hline 
   Reedy left & $Seg$ & Segal coCartesian & Segal coCartesian &  $SegcoCart$ \\ \hline 
   Reedy right & $Seg$ & Segal Cartesian & Segal Cartesian & $SegCart$ \\ \hline 
   Reedy left & $CSS$ & coCartesian & coCartesian & $coCart$ \\ \hline 
   Reedy right & $CSS$ & Cartesian &  Cartesian & $Cart$ \\ \hline 
   Reedy left & $Kan$ & left & covariant & $cov$ \\ \hline 
   Reedy right & $Kan$ & right & contravariant & $contra$
  \end{tabular}

 \end{center}

 We now have following results using the table above.
 
  \begin{theone} \label{the:segal Cartesian model structure}
  	(\cref{The Localized Reedy contravariant model structure}) 
  There is a unique simplicial combinatorial left proper model structure on bisimplicial spaces over $X$, called the 
  $(\M)$-model structure and denoted by $(\sss_{/X})^{(\D)}$ satisfying following conditions.
  \begin{enumerate}
   \item The fibrant objects are the $(\F)$-fibrations over $X$.
   \item Cofibrations are monomorphisms.
   \item A map of bisimplicial spaces $Y \to Z$ over $X$ is a weak equivalence if 
   $$\Map_{/X}(Z,W) \to \Map_{/X}(Y,W)$$ 
   is a Kan equivalence for every $(\F)$-fibration $W \to X$.
   \item A weak equivalence ($(\F)$-fibration) between fibrant objects is a level-wise equivalence (biReedy fibration). 
  \end{enumerate}
 \end{theone} 
 
  \begin{propone} 
  	(\cref{propone:local relative gives local})
 	The following adjunction
 	\begin{center}
 		\adjun{(\sss_{/X})^{(\D)}}{(\sss_{/X})^{Diag-(\V)}}{id}{id}
 	\end{center}
 	is a Quillen adjunction, which is a Quillen equivalence whenever $X$ is homotopically constant. 
 	Here the left hand side has the $(\M)$ model structure and the left hand side 
 	has the induced diagonal $(\V)$-model structure over the base $X$.
 \end{propone}

\begin{theone} (\cref{the:simplicial groth construction localized})
	Let $\C$ be a small category. Then the adjunctions defined in \cref{the:simplicial groth construction}
	\begin{center}
		\begin{tikzcd}[row sep=0.5in, column sep=0.9in]
			\Fun(\C,\ss^{(\V)})^{proj} \arrow[r, shift left = 1.8, "\ssint_\C"] & 
			(\sss_{/N\C})^{(\D)} \arrow[l, shift left=1.8, "\ssH_\C", "\bot"'] \arrow[r, shift left=1.8, "\ssbT_\C"] &
			\Fun(\C,\ss^{(\V)})^{proj} \arrow[l, shift left=1.8, "\ssbI_\C", "\bot"']
		\end{tikzcd}
	\end{center}
	are simplicial Quillen equivalences. Here the middle has the $(\M)$-model structure
	and the two sides have the projective model structure on the $(\V)$-model structure. 
\end{theone}

\begin{theone}
	(\cref{cor:localized left}) 
	Let $p: L \to X$ be a biReedy fibration.
	Then the following are equivalent:
	\begin{enumerate}
		\item $p$ is an $(\F)$-fibration.
		\item $p$ is an $(\R)$-fibration and the simplicial space $\Val(L)$ is local with respect to $(\V)$.
		\item $p$ is an $(\R)$-fibration and the simplicial spaces $\Val_k(L)$ are local with respect to $(\V)$ for all $k \geq 0$.
		\item For every map $\sigma: F(n) \times \Delta[l] \to X$, the pullback map $\sigma^*p: \sigma^*L \to F(n) \times \Delta[l]$
		is an $(\F)$-fibration.
		\item For every map $\sigma: F(n) \to X$, the pullback map $\sigma^*p: \sigma^*L \to F(n)$
		is an $(\F)$-fibration.
		\item $p$ is an $(\R)$-fibration and for every point $\{x\}: F(0) \to X$ the fiber $\Fib_xL$ is fibrant in the diagonal $(\V)$-model structure.
		\item $p$ is an $(\R)$-fibration and for every point $\{x\}: F(0) \to X$ the fiber $\Val(\Fib_xL)$ is fibrant in the $(\V)$-model structure.
	\end{enumerate}
\end{theone}

 \begin{theone} (\cref{theone:local equiv between fibrants})
	Let $L$ and $M$ be two $(\F)$-fibrations over $X$.
	Let $g:L \to M$ be a map over $X$. Then the following are equivalent.
	\begin{enumerate}
		\item $g: L \to M$ is a biReedy equivalence.
		\item $\Val(g): \Val(L) \to \Val(M)$ is a Reedy equivalence.
		\item For every $\{ x \}: F(0) \to X$, the map  $\Fib_x \Val(L) \to \Fib_x \Val(M)$ is a Reedy equivalence of bisimplicial spaces.
		\item For every $\{x \}: F(0) \to X$, the map $\Fib_x(L) \to \Fib_x (M)$ is a diagonal Reedy equivalence of bisimplicial spaces.
	\end{enumerate}
\end{theone}

For the next proposition we need following construction.
Let $p: W \to X$ be an $(\R)$-fibration. Then we can construct following diagram
\begin{equation} \label{eq:fib rep cor}
	\begin{tikzcd}[row sep=0.2in, column sep=0.3in]
		W_\bullet \arrow[drr, "p"'] \arrow[rr, hookrightarrow, "i", "\simeq"'] & & 
		\tilde{W}_\bullet \arrow[d, "\tilde{p}" twoheadrightarrow] \arrow[rr, "j", "\simeq"'] 
		& & \hat{W} \arrow[dll, twoheadrightarrow, "\hat{p}"] \\
		& & X & &  	
	\end{tikzcd}.	 
\end{equation}
Here the first map is the level-wise functorial factorization of the simplicial object in $(\ss_{/X})^{(\V)}$, the $(\V)$-model structure. Moreover, let $\hat{p}: \hat{L} \to X$ be a biReedy fibrant replacement over $X$.

\begin{propone} (\cref{prop:localized equiv between Reedy left})
	Let $p: L \to X$, $q: M \to X$ be an $(\R)$-fibrations and let $f:L \to M$ be a map over $X$. Then the following are equivalent.
	\begin{enumerate}
		\item $f$ is an $(\M)$-equivalence.
		\item The map $\hat{f}: \hat{L} \to \hat{M}$ constructed in \ref{eq:fib rep cor} is a biReedy equivalence.
		\item The map 
		$$\Val(\hat{f}): \Val(\hat{L}) \to \Val(\hat{M})$$ 
		constructed in \ref{eq:fib rep cor} is a Reedy equivalence.
		\item The map
		$$\Val(f): \Val(L) \to \Val(M)$$
		is a $(\V)$-equivalence.
		\item For every object $\{x\}: F(0) \to X$, the induced map on fibers 
		$$\Val(\Fib_x L) \to \Val(\Fib_x M)$$
		is a $(\V)$-equivalence.
		\item For every object $\{x\}: F(0) \to X$, the induced map on fibers 
		$$\Fib_x L \to \Fib_x M$$
		is a diagonal $(\V)$-equivalence.
	\end{enumerate}
\end{propone}

 \begin{theone} (\cref{the:recognition principle for localized Reedy right equivalences})
	A map $g:Y \to Z$ of bisimplicial spaces over $X$ is an $(\M)$-equivalence if and only if for each map $\{x\}: F(0) \to X$, the induced map 
	$$ Y \underset{X}{\times} \LFib(R_x) \to Z \underset{X}{\times} \LFib(R_x)$$
	is an equivalence in the diagonal $(\V)$-model structure. Here $R_x$ is a choice of right fibrant replacement of the map $\{x\}$.
\end{theone}

\begin{theone} (\cref{the:loc Ree contra invariant under CSS}) \label{the:cart invariant}
	Let $g: X \to Y$ be a map of simplicial spaces. Then the adjunction
	\begin{center}
		\adjun{(\sss_{/X})^{(\D)}}{(\sss_{/Y})^{(\D)}}{g_!}{g^*}
	\end{center}
	is a Quillen adjunction, which is a Quillen equivalence whenever $g$ is a CSS equivalence.
	Here both sides have the $(\M)$-model structure.
\end{theone}

\begin{theone} (\cref{the:pulling back along loc Ree right preserves loc Ree cov})
	Let $p:V \to X$ be a dual of an $(\R)$-fibration over $X$. The induced adjunction
	\begin{center}
		\adjun{(\sss_{/X})^{(\D)}}{(\sss_{/X})^{(\D)}}{p_!p^*}{p_*p^*}
	\end{center}
	is a Quillen adjunction. Here both sides have the $(\M)$-model structure.
\end{theone}

 \bibliographystyle{alpha}
 \bibliography{main}

\end{document}